\documentclass[a4paper,12pt]{article}
\usepackage[utf8]{inputenc}
\usepackage{fullpage}
\usepackage{amsmath, amssymb, amstext, amsfonts, amsthm, bbm, array, enumerate,scalerel}
\usepackage{mathrsfs}%\mathscr schoen, \mathcal.. nicht so schoen
\usepackage{nicefrac}
\usepackage{xcolor}
\usepackage[english]{babel}
\usepackage{tikz}
\usepackage{stmaryrd}
\usepackage{tikz-3dplot}
\usetikzlibrary{patterns}
\usetikzlibrary{plotmarks}
\usepackage[round]{natbib}
\usepackage{enumitem}

\newtheorem{theorem}{Theorem}[section]
\newtheorem{lemma}[theorem]{Lemma}
\newtheorem{corollary}[theorem]{Corollary}

\theoremstyle{definition}
\newtheorem{definition}[theorem]{Definition}
\newtheorem{example}[theorem]{Example}

\theoremstyle{remark}
\newtheorem{remark}[theorem]{Remark}

\numberwithin{equation}{section}
\numberwithin{theorem}{section}

\newcommand{\E}{\mathbb{E}}

\newcommand{\Xcal}{{\mathcal X}}
\newcommand{\Acal}{{\mathcal A}}
\newcommand{\Bcal}{{\mathcal B}}
\newcommand{\Mcal}{{\mathcal M}}
\newcommand{\s}{{\mathcal S}}
\newcommand{\Dcal}{{\mathcal D}}
\newcommand{\Scal}{{\mathcal S}}
\newcommand{\e}{{\varepsilon}}
\DeclareMathOperator{\supp}{supp}
\newcommand{\Fcal}{{\mathcal F}}
\renewcommand{\P}{{\mathbb P}}
\newcommand{\Pol}{\textup{Pol}}
\newcommand{\fdot}{\,\cdot\,}

\newcommand{\R}{\mathbb{R}}

\newcommand{\N}{\mathbb{N}}

\newcommand{\diag}{\operatorname{diag}}

\RequirePackage[colorlinks,citecolor=blue,urlcolor=blue, linkcolor=blue]{hyperref}

\newcommand\mydots{\hbox to 1em{.\hss.\hss.}}
\title{Measure-valued affine and polynomial diffusions}

\author{Christa Cuchiero\thanks{Vienna University, Department of Statistics and Operations Research, Data Science @ Uni Vienna, Kolingasse 14-16, 1090 Wien, Austria, christa.cuchiero@univie.ac.at
} \quad Francesco Guida \thanks{University of Trento and University of Verona, Department of Mathematics, Via Sommarive 14, 38123 Povo, Italy, francesco.guida@unitn.it
} \quad  Luca di Persio \thanks{University of Verona, Department of Computer Science, Strada le Grazie 15, 37134 Verona, Italy, luca.dipersio@univr.it} \\ Sara Svaluto-Ferro \thanks{University of Verona, Department of Economics, Via Cantarane 24, 37129 Verona, Italy, sara.svalutoferro@univr.it \newline
The first and the fourth author gratefully acknowledge financial support 
through grant Y 1235 of the START-program.
}}
\date{}

\begin{document}
\maketitle

\begin{abstract}
We introduce a class of measure-valued processes, which -- in analogy to their finite dimensional counterparts -- will be called measure-valued polynomial diffusions.  We show the so-called moment formula, i.e.~a representation of the conditional marginal moments via a system of finite dimensional linear PDEs. Furthermore, we characterize the corresponding infinitesimal generators and obtain a representation analogous to polynomial diffusions on $\mathbb{R}^m_+$,  in cases where their domain is large enough. In general the infinite dimensional setting allows for richer specifications strictly beyond this representation. As a special case we recover  measure-valued affine diffusions, sometimes also called Dawson-Watanabe superprocesses.\\
From a mathematical finance point of view the polynomial framework is especially attractive as it allows to transfer the most famous finite dimensional models, such as the Black-Scholes model, to an infinite dimensional measure-valued setting. 
We outline in particular the applicability of our approach for term structure modeling in energy markets.
\end{abstract}

\textbf{Keywords:} measure-valued processes;  polynomial and affine diffusions; Dawson-Watanabe type superprocesses;  martingale problem; maximum principle; HJM term structure modeling; energy markets

\textbf{AMS MSC 2020:} 60J68, 60G57
\tableofcontents

\section{Introduction}
We introduce a class of (non-negative) measure-valued processes, which we call in analogy to their finite dimensional counterparts \emph{measure-valued polynomial diffusions}. In spirit of \cite{CLS:19}, where the focus was on probability measures, we now transfer the defining property of finite dimensional polynomial processes considered e.g. in \cite{CKT:12, FL:16, LP:17, CLS:18} to the state space of (non-negative) measures.

This state space has high relevance for applications since stochastic modeling of infinite dimensional non-negative quantities occurs in many areas. This applies especially  to biology and population genetics where measure-valued processes have played a crucial role since the 1970's, when branching Brownian motion was introduced by Henry McKean, see \cite{MCKean:75}. Also the polynomial property occurs in this field often naturally, see,  e.g., \cite{BCGK:16}  for an example of a recently investigated two-dimensional polynomial process.

Similarly to finite dimensional polynomial processes on non-negative state spaces which have been applied in many areas of finance, like credit risk, stochastic volatility, or life insurance liability modeling (see, e.g., \cite{AF:20, AFP:18, AMLB:20, BZ:16}), measure-valued polynomial processes are also very well suited for modeling in mathematical finance. Indeed, infinite dimensional non-negative processes are of particular interest in view of term structures as we shall illustrate below by analyzing forward price models in energy markets. Another instance of measure-valued processes concerns Markovian lifts of stochastic Volterra processes (see, e.g., \cite{A:19, CT:19,CT:20}), which became recently very popular in view of rough volatility modeling initiated by \cite{GJR:18}. Stochastic portfolio theory, introduced in \cite{F:00} (see also \cite{FK:09}), is another area of applications where one deals with potentially high or infinite dimensional markets. Hence a measure-valued analog of the finite dimensional setup considered e.g.~in \cite{FK:05, C:19} is of interest to study asymptotic relative arbitrages. Measure-valued processes can also be used as state processes for stochastic optimal control problems with applications including model-independent derivatives pricing and two player games with asymmetric information (see, e.g., \cite{CKLS:21}). 

The advantage of using measure-valued processes instead of function-valued ones  is that many spatial stochastic processes do not fall
into the framework of stochastic partial differential equations (SPDEs).
Indeed, it is often easier to establish existence of a measure-valued process than of an analogous stochastic partial differential equation, say in some Hilbert space, which would for instance
correspond to its Lebesgue density, but which does not necessarily exist.
For an introduction, important concepts and results related to measure-valued processes we refer to \cite{E:00, D:93} and \cite{L:10}. For recent results in the probability measure valued case on generalized Wasserstein spaces see~\cite{LS:19}.

Beyond the general suitability of measure-valued processes for infinite dimensional dynamic modeling, the current polynomial framework allows for high tractability and the possibility to transfer many famous finite dimensional models, such as the Black-Scholes model or the Feller diffusion, to an infinite dimensional measure-valued setting. It offers
a unified treatment of a large class of processes that includes these well-known examples but also goes far beyond them as we shall illustrate in Theorem \ref{main1}. 

Let us now describe the precise setup in more detail. We denote by $X$  measure-valued polynomial diffusions taking values in the space of (non-negative) measures on a locally compact Polish space $E$. We define them as solutions of martingale problems with continuous trajectories for certain operators $L$ acting on classes of cylindrical functions, i.e.~functions $f$ of the form
\[
\textstyle f(\nu)=\phi\left(\int_E g_1(x)\nu(dx),\ldots,\int_E g_m(x)\nu(dx)\right),
\]
where $\phi \in C^{\infty}(\mathbb{R}^m)$, $g_1,\ldots,g_m$ are continuous and bounded, and the argument $\nu$ is a (non-negative) measure.  The defining property of a measure-valued polynomial diffusion is that when $L$ is applied to a cylindrical polynomial, i.e. $\phi$ is a polynomial on $\mathbb{R}^m$, then $Lf$ is again a polynomial with the same degree as $f$. Precise definitions are given in Section~\ref{sec_pol_op}. For  related concepts in general Banach spaces we refer to the inspiring paper 
\cite{BDK:20}.

As a consequence of the defining property of polynomial diffusion we can then prove the following results:

\begin{itemize}
\item We show that the conditional marginal moments of measure-valued polynomial diffusions can be represented via a system of \emph{finite dimensional linear PDEs}, whose (maximal) spatial dimension corresponds to the degree of the  moment, see Theorem \ref{thm1}. This means that the
so-called \emph{moment formula} holds and that the tractability of finite dimensional polynomial processes can be preserved in this setting. Indeed, for polynomial terminal conditions the infinite dimensional Kolmogorov backward equation can be reduced to finite dimensional system of PDEs.

\item We give necessary and sufficient conditions for the existence of measure-valued polynomial diffusions by analyzing the corresponding martingale problem. When the domain of the respective operators is large enough, we obtain a representation of the drift and diffusion part analogous to polynomial diffusions on $\mathbb{R}^m_+$, analyzed in \cite{FL:16}, see Theorem \ref{mainthm}. Otherwise the infinite dimensional setting allows for much more generality by going strictly beyond this representation, see Theorem \ref{main1}.

\item One key ingredient in these proofs is
a characterization of the positive maximum principle of the respective operators. To this end we also obtain new optimality conditions which are applicable for all martingale problems on the space of (non-negative) measures, see Section \ref{IIstool}.

\item By specifying affine type operators in Section \ref{sec:affine} we recover Dawson-Watanabe-type superprocesses as the affine subclass of polynomial diffusions. 
In this case, additionally to the moment formula, the Laplace transform of the process' marginals is exponentially affine in the initial state. As well-known the corresponding characteristic exponent can then be computed by solving the associated Riccati partial differential equations (see, e.g., \cite{L:10}).
Let us here also mention that infinite dimensional affine
processes, however with values in Hilbert spaces, have recently been studied in \cite{STW:20, CKK:21}.

\item  We also show under which conditions exponential moment exist, so that uniqueness in law holds and the corresponding martingale problems are well-posed, see Theorem \ref{thm5}. This condition is satisfied by the subclass of measure-valued affine diffusions, whose marginal laws are then characterized by the solution of the associated Riccati partial differential equations.

\item We provide several concrete specifications in Section \ref{sec:applications}. Indeed, we consider as underlying space $E$ a finite set of points and recover thus the results of \cite{FL:16}. Moreover, we analyze in detail the case when $E \subseteq \mathbb{R}$, exploiting the characterization of strongly continuous positive groups provided in \cite{arendt:86}, which allows us to specify a concrete and easily applicable form of the infinitesimal generator. We also connect our findings with measure-valued analogs of Pearson diffusions (see, e.g., \cite{FS:08}) and the Black-Scholes model.
\end{itemize}

This shows in particular how to construct measure-valued versions of the  most famous finite dimensional models in finance and hence why 
measure-valued affine and polynomial diffusions qualify as tractable models for all high dimensional financial markets potentially  involving an infinite number of assets. As already mentioned above, this concerns especially term structures arising from fixed income markets with possibly multiple yield curves (see, e.g.,  \cite{CFG:16}), equity derivatives (see \cite{KK:15} and the references therein) or forward contracts in energy markets (see \cite{BBK:08, BK:14}).
Time to maturity (or the strike dimension in case of equity derivatives) then takes the role of the spatial structure. 

Our results can most notably be used for energy market modeling, in particular for electricity and gas markets, whose essential products are futures and forward contracts as well as options written on these.
We  concentrate on futures (also called swaps) with delivery over a
 time interval $[\tau_1, \tau_2]$. Their price at time $t \geq 0$ is denoted by $F (t, \tau_1 , \tau_2 )$ at $t \leq  \tau_1$. Following \cite{BBK:08}, $F (t, \tau_1 , \tau_2 )$ can be written as a weighted  integral of instantaneous forward prices $f(t,u)$ with delivery at time  $\tau_1 \leq u \leq \tau_2$, i.e.
\[
F (t, \tau_1 , \tau_2 )=\int_{\tau_1}^{\tau_2} w(u,\tau_1, \tau_2) f(t,u) du,
\]
where $w(u,\tau_1, \tau_2) $ denotes some weight function.
The crucial reason for using measure-valued processes in  electricity and gas modeling now comes from the fact that
there is no trading with the instantaneous forwards $f(t,u)$ for obvious reasons.
Thus, rather than  using $f(t,u) du$ in the expression of the future prices, we can also use a measure, which cannot necessarily be evaluated pointwise.
In the companion paper \cite{CGP:21} we establish a 
\emph{Heath-Jarrow-Morton} (HJM) approach (see \cite{HJM:92}) with measure-valued processes instead of processes taking values in some Hilbert space of functions.
We then obtain a HJM-drift condition that restricts the choice of the process since the drift part is completely determined, but we are free to specify its martingale part as long as we do not leave the state space of (non-negative) measures. Here, the full specification 
of affine and polynomial diffusions provided in Section \ref{sec_ex} comes into play, since it allows us to obtain a rich and flexible class of models while preserving tractability. Indeed, the martingale part given by $Q_1$ and $Q_2$ in Theorem \ref{main1} depends on  continuous functions satisfying certain admissibility conditions. For calibration and pricing purposes these functions can then be parametrized by neural networks. This in turn leads to fast calibration procedures where on the one hand the analytic tractability coming from the affine and polynomial nature and on the other hand (stochastic) gradient descent methods for neural networks can be exploited.

The remainder of the paper is organized as follows. In Section \ref{IIs1}
we introduce frequently used notation and  basic definitions. In Section \ref{IIsec:2} we define polynomials and cylindrical functions of measure arguments, while in Section \ref{IIstool} we prove optimality conditions for these function, which we need to apply the  positive maximum principle later on. In Section \ref{sec_pol_op} we  introduce the notion of polynomial operators in the current setting and characterize them as second order differential operator with affine drift and quadratic diffusion part. Section \ref{sec_ex} contains the main results, i.e.~the moment formula, sufficient conditions for existence and uniqueness and a full characterization when the domain of the infinitesimal generator is rich enough. Section \ref{sec:affine} is dedicated to the analysis of the affine subclass and Section \ref{sec:applications} concludes with exemplary specifications. In Appendix \ref{app:A} we prove necessary conditions for the positive maximum principle, that  in turn implies  existence of solutions to the martingale problem, see Appendix \ref{app_existence}. In Appendix \ref{app:C} we  recall a characterization of generators of strongly continuous positive groups proved in \cite{arendt:86}.

\subsection{Notation and basic definitions}\label{IIs1}

%Similarly to the terminology applied in \cite{CLS:19}, we shall use the following notation.
Throughout this article, $E$ is a locally compact Polish space endowed with its Borel $\sigma$-algebra.  

\begin{itemize}
\item $M_+(E)$ denotes the set of finite non-negative measures on $E$, 
and $M(E)=M_+(E)-M_+(E)$ the space of signed measures of the form $\nu_+-\nu_-$ with $\nu_+,\nu_-\in M_+(E)$. We usually leave ``non-negative'' away and just say finite measures for  $M_+(E)$.
Both $M(E)$ and $M_+(E)$ are equipped with the topology of weak convergence, which turns $M_+(E)$ into a Polish space. For $\mu,\nu\in M(E)$ we write $\mu\leq\nu$ if $\nu-\mu\in M_+(E)$ and $|\nu|$ for $\nu_++\nu_-$.

\item $C(E)$, $C_b(E)$, $C_0(E)$, $C_c(E)$ denote, respectively,  the spaces of  continuous, bounded, vanishing at infinity, and compactly supported real functions on $E$. We equip the  latter three with the topology of uniform convergence, and denote by $\|\fdot\|$ the supremum norm.

\item If $E$ is non-compact, then $E^\Delta=E\cup\{\Delta\}$ denotes the one-point compactification, i.e.~a compact Polish space. If $E$ is compact we write $E^\Delta=E$. We also define
\begin{align} \label{eq:C_Delta}
C_\Delta(E^k):=\big\{f|_{E^k}\ :\ f\in C((E^\Delta)^k)\big\},
\end{align}
which is a closed subspace of $C_b(E^k)$. The spaces $C_\Delta(E)$ and $C(E^\Delta)$ can be identified, 
and we sometimes regard elements of the former as elements of the latter, and vice versa. When $E$ is compact,  $C(E)=C_b(E)=C_0(E)=C_c(E)=C_\Delta(E)$ holds true and we thus simply write $C(E)$. Note that the constant function $1$ lies in $C_\Delta(E)$, but of course not in $C_0(E)$. This is one reason why we introduce the function space $C_\Delta(E^k)$.

\item One important underlying space, taking in the above definition the role of $E$, will be $M_+(E)$ if $E$ is compact, and $M_+(E^{\Delta})$ otherwise.
Recall that $M_+(E)$ is locally compact when $E$ is compact (see, e.g., \cite[Remark 1.2.3]{L:70}), which however does not hold true when $E$ is non-compact. Therefore we shall consider $M_+(E^{\Delta})$  in the non-compact case. Identifying $E^{\Delta}$ with $E$, when $E$ is compact, we just write $M_+(E^{\Delta})$ for both cases.
In order to distinguish between the different one-point compactifications,  we denote the one-point compactification of  $M_+(E^{\Delta})$ by $M^{\mathfrak{\Delta}}_+(E^{\Delta})$ and identify  it with measures of infinite mass. 
The function space $C_{\mathfrak{\Delta}}(M_+(E^{\Delta}))$ is defined analogously to \eqref{eq:C_Delta} and can be identified with $C(M^{\mathfrak{\Delta}}_+(E^{\Delta}))$.

\item $\widehat C_\Delta (E^k)$ is the closed subspace of $C_\Delta (E^k)$ consisting of symmetric functions $f$, i.e., 
$f(x_1,\ldots,x_k)=f(x_{\sigma(1)},\ldots,x_{\sigma(k)})$ for all $\sigma\in \Sigma_k$, the permutation group of $k$ elements.
$\widehat C_0(E^k)$ and $\widehat C(E^k)$ are defined similarly.
 For any $g\in\widehat C_\Delta (E^k), h\in \widehat C_\Delta (E^\ell)$ we denote by
$g\otimes h\in\widehat C_\Delta(E^{k+\ell})$ the symmetric tensor product, given by

\begin{equation} \label{IIeqn17}
\begin{aligned}
(g\otimes h) & (x_1,\ldots,x_{k+\ell}) \\
&= \frac{1}{(k+\ell)!} \sum_{\sigma\in\Sigma_{k+\ell}} g\big(x_{\sigma(1)},\ldots,x_{\sigma(k)}\big)h\big(x_{\sigma(k+1)},\ldots,x_{\sigma(k+\ell)}\big).
\end{aligned}
\end{equation}
 We emphasize that only symmetric tensor products are used in this paper.
\item Throughout the paper we let
$$
D \subseteq C_\Delta (E)
$$
be a dense linear subspace containing the constant function 1. We set $D^{\otimes 2}=D\otimes D:=\text{span}\{g\otimes g\ :\ g\in D\}$. This generalizes of course to $D^{\otimes k}$ for $k > 2$.
\end{itemize}

Two key notions that we shall often use are the \emph{positive maximum principle,} and the \emph{positive minimum principle}  for certain linear operators.
\begin{definition} \label{def1}
Fix a Polish space $\Xcal$ and a subset $\s\subseteq \Xcal$. Moreover, let $D(\Acal)$ be a linear subspace of $C(\Xcal)$.
\begin{itemize}
\item 
An operator $\Acal$ 
with domain $D(\Acal)$  is said to satisfy the positive maximum principle on $\s$ if 
$$\text{$f\in D(\Acal)$, $x\in \s$, $\sup_\s f=f(x)\geq0\quad$ implies $\quad\Acal f(x)\leq0$.}$$

\item An operator $\Acal$ 
with domain $D(\Acal)$ is said to satisfy the positive minimum principle on $\s$ if 
\[
0 \leq g  \in D(\Acal),\, x \in \s,\, \inf_{\s}g=g(x)=0\quad \text{implies} \quad \mathcal{A}g(x)\geq 0.
\]
\end{itemize}
\end{definition}

\begin{remark}\label{rem1}
Note that the positive maximum principle implies the positive minimum principle. Indeed, let $ g\geq 0$ with $\inf_\s g=g(x)=0$ and set $f=-g$ (which is possible since $D$ is a linear space). Then the positive maximum principle yields $\mathcal{A}f(x)=-\mathcal{A}g(x)\leq 0$.

Moreover, let $\Acal$ be an operator with domain $D(\Acal)$ satisfying the positive minimum principle on $\Scal$ and suppose that $1\in D(\Acal)$. 
Then $\Acal g =\Bcal g+mg$ for some map $m$ and some operator $\Bcal$ satisfying $\Bcal1=0$ and the positive maximum principle on $\Scal$. $\Bcal$ and $m$ can be explicitly constructed by setting $\Bcal g(x)=\Acal(g-g(x))(x)$ and $m(x):=\Acal 1(x)$.
\end{remark}

These notions shall play an important role throughout the paper. Indeed, the positive maximum principle (combined with conservativity) is essentially equivalent to the existence of an $\Scal$-valued solution to the martingale problem for $\Acal$, see \cite[Theorem 4.5.4]{EK:09}. Here it is crucial that $\Scal$ is locally compact.  In our setting the most important case is $\Scal= M_+(E^\Delta)$ which is locally compact.

The positive minimum principle will be used to characterize certain operators appearing in the generator of measure-valued polynomial diffusions. Note here that for generators of strongly continuous semigroups on $C(\mathcal{X})$ the  positive minimum principle is equivalent to generating a \emph{positive} semigroup, if $\mathcal{X}$ is compact, see \cite[Theorem B.II.1.6]{arendt:86}.

Finally, we indicate bounded pointwise limits by $\text{bp-}\lim$ in the sense of \cite[Appendix 3]{EK:09}.

\section{Polynomials and cylindrical functions of measure arguments} \label{IIsec:2}

In this section we recall the notion of polynomials of measure arguments similarly as in \cite{CLS:19}. 

\subsection{Monomials and polynomials}\label{IIs21}

A \emph{monomial} on $M(E)$ is defined via
\[
\langle g, \nu^k \rangle = \int_{E^k} g(x_1,\ldots,x_k) \nu(dx_1) \cdots \nu(dx_k)
\]
for some $k\in\N_0$, where $g\in \widehat C_\Delta (E^k)$ is referred to as the \emph{coefficient} of the monomial; see, e.g., \cite[Chapter 2]{D:93}. We identify $\widehat C_\Delta (E^0)$ with $\R$, so that for $k=0$ we have $\langle g,\nu^0\rangle=g\in\R$. The map $\nu\mapsto\langle g,\nu^k\rangle$ is clearly homogeneous of degree~$k$, and $g\mapsto\langle g,\nu^k\rangle$ is linear. Furthermore, the identity $\langle g,\nu^k\rangle \langle h,\nu^\ell\rangle = \langle g\otimes h,\nu^{k+\ell}\rangle$, holds true, where the symmetric tensor product $g\otimes h$ is defined in \eqref{IIeqn17}.

A polynomial on $M(E)$ is now defined as a (finite) linear combination of monomials,
\begin{equation} \label{IIeq:p(nu)}
p(\nu) = \sum_{k=0}^m \langle g_k,\nu^k\rangle,
\end{equation}
with coefficients $g_k\in\widehat C_\Delta (E^k)$, and 
we shall denote the coefficients vector $(g_0, \ldots, g_m)$ of a polynomial by $\vec{g}=(g_0, \ldots, g_m) \in \bigoplus_{k=0}^m \widehat{C}_\Delta (E^k)$.

 The degree of a polynomial $p$, denoted by $\deg(p)$, is the largest $k$ such that $g_k$ is not the zero function, and $-\infty$ if $p$ is the zero polynomial. The representation~\eqref{IIeq:p(nu)} is unique; see Corollary 2.4 in \cite{CLS:19}.\\

The following example shows that the notion of polynomials coincides with the usual one on $\mathbb{R}^m$ when $E$ is a finite set with $m$ elements.

\begin{example}\label{IIex4}
Let $E=\{1,\ldots,m\}$ be a finite set. Then every element $\nu\in M(E)$ is of the form
\[
\nu = c_1\delta_1 + \cdots + c_m \delta_m, \qquad (c_1,\ldots,c_m)\in\R^m,
\]
where $\delta_i$ is the Dirac mass concentrated at $\{i\}$. Monomials take the form
\[
\langle g,\nu^k\rangle = \sum_{i_1,\ldots,i_k} g(i_1,\ldots,i_k)\, c_{i_1}\cdots c_{i_k},
\]
where the summation ranges over $E^k=\{1,\ldots,m\}^k$. Therefore, as $g$ can be any symmetric function on $E^k$, we recover all homogeneous polynomials of total degree $k$ in the $m$ variables $c_1,\ldots,c_m$. 
\end{example}

In the following we define the function space of polynomials on $M(E)$.

\begin{definition} \label{IID:Pol(M(E))}
Let
\[
P := \left\{\nu\mapsto p(\nu) \colon \text{$p$ is a polynomial on $M(E)$} \right\}
\]
denote the algebra of all polynomials on $M(E)$ regarded as real-valued maps, equipped with pointwise addition and multiplication.
\end{definition}

The subsequent lemma asserts that 
polynomials on $M(E)$ can be uniquely extended to polynomials on $M(E^{\Delta})$, which we shall apply regularly. For its proof we refer to Lemma 2.3.1 in \cite{CLS:19}

\begin{lemma} \label{IIL:Psmooth}
Each $p\in P$ is continuous on $M_+(E)$, sequentially continuous on $M(E)$, and can be uniquely extended to a polynomial on $M(E^\Delta)$.\footnote{It can be shown that sequential continuity cannot be strengthened to continuity.}
\end{lemma}

\subsection{Directional derivatives}

The notion of derivatives of a function on $M(E)$ that we  apply throughout the paper is a 
directional one. More precisely, 
a function $f\colon M(E)\to\R$ is called differentiable at $\nu$ in direction $\delta_x$ for $x\in E$ if
\[
\partial_x f(\nu) := \lim_{\varepsilon\to0} \frac{f(\nu+\varepsilon\delta_x)-f(\nu)}{\varepsilon}
\]
exists. We write $\partial f(\nu)$ for the map $x\mapsto\partial_xf(\nu)$, and we use the notation
\[
\partial^k_{x_1x_2\cdots x_k} f(\nu) := \partial_{x_1}\partial_{x_2}\cdots \partial_{x_k} f(\nu)
\]
for iterated derivatives. We write $\partial^k f(\nu)$ for the corresponding map from $E^k$ to $\R$. Observe that for a linear map $p\in P$ of the form $p(\nu)=\langle g, \nu\rangle$ we get
$$\partial_x p(\nu)=\lim_{\e\to0}\frac{1}{\e}\int g(y) \e \delta_x(dy)=g(x)$$ for each $x\in E$.

\subsection{Cylindrical polynomials with regular coefficients}\label{sec:polycyl}

In order to be able to consider certain differential operators later on,  let us introduce subspaces of polynomials with more regular coefficients. 
Let $\Pol(\R^m)$ denote the set of polynomials on $\R^m$ and recall that
$
D \subseteq C_\Delta (E)
$
is a dense linear subspace containing the constant function 1. We then call \emph{cylindrical polynomial} an element of the set
$$
P^{D}:=\Big\{\nu \mapsto p(\nu)=\phi(\langle g_1, \nu \rangle, \ldots, \langle g_m, \nu \rangle)
\colon   \phi \in \Pol(\R^m),\, g_k \in D,\,  m\in \mathbb{N}_0\Big\}.
$$
Note that $P^D$ is the subalgebra of $P$ consisting of all (finite) linear combinations of the constant polynomial and ``rank-one'' monomials $\langle g\otimes\cdots\otimes g,\nu^k\rangle=\langle g,\nu\rangle^k$ with $g\in D$.

Since for $g_{k_1} \in \widehat{C}_{\Delta}(E^{k_1})$ and $g_{k_2} \in \widehat{C}_{\Delta}(E^{k_2})$,  it holds 
$ \langle g_{k_1}, \nu^{k_1} \rangle  \langle g_{k_2}, \nu^{k_2} \rangle =  \langle g_{k_1} \otimes g_{k_2}, \nu^{k_1+k_2} \rangle$, we can equivalently write 
\[
P^D=\left\{ \nu \mapsto \sum_{k=0}^m \langle g_k, \nu^k \rangle \colon m\in \mathbb{N}_0, \ g_k\in D^{\otimes k}\right\},
\]
where $D^{\otimes k}$ was introduced right after \eqref{IIeqn17}.
For the sake of completeness, let us restate parts of Lemma~2.7 of \cite{CLS:19}, showing that the derivative of $x\mapsto\partial_x p(\nu)$ has the same regularity as the coefficients of~$p$, the latter being our main motivation to introduce $P^D$.

\begin{lemma}\label{IIIlem1}
For any $p\in P^D$ and $\nu\in M(E)$, we have $\partial^k p(\nu)\in D^{\otimes k}$.  
\end{lemma}

\subsection{Cylindrical functions}\label{sec:funcyl}

Similarly to cylindrical polynomials we also consider cylindrical functions. Define 
$$
F^{D}=\Big\{ \nu \mapsto f(\nu)=\phi(\langle g_1, \nu \rangle, \ldots, \langle g_m, \nu \rangle)
\colon,   \phi \in C^{\infty}_0(\mathbb{R}^m),  g_k \in D,\,  m\in \mathbb{N}_0\Big\}.
$$
Since concatenations of continuous functions are continuous, cylindrical functions satisfy the following lemma.

\begin{lemma} \label{Fcont}
 Each $f\in F^D$ is continuous on $M_+(E)$, sequentially continuous on $M(E)$, and can be uniquely extended to  $M(E^\Delta)$.
\end{lemma}

For further use we now define the restriction of $F^D$ to $M_+(E^\Delta)$, i.e.
\[
F^D(M_+(E^\Delta)):=\{f|_{M_+(E^\Delta)}\colon f\in F^D\}.
\]
Observe that functions in $F^D$ are first extended to $M(E^\Delta)$ using Lemma \ref{Fcont} and then restricted  to $M_+(E^\Delta)$. Since elements in $F^D(M_+(E^\Delta))$ do not need to be compactly supported we also define
$$F^D_c:=F^D(M_+(E^\Delta)) \cap C_c(M_+(E^\Delta)).$$

\begin{lemma}\label{IIIlem1mod}
For any $f\in F^D$ and $\nu\in M(E)$, we have $\partial^k f(\nu)\in D^{\otimes k}$.  
Moreover $F^D_c$ is dense in $C_0(M_+(E^\Delta))$. 
\end{lemma}

\begin{proof}
For $f(\nu):=\langle g,\nu\rangle$  we have $\partial^k f(\nu)= g^{\otimes k}\in D^{\otimes k}$. The first part follows by the chain rule.
Since $F^D_c$ is
a point separating algebra that vanishes nowhere, the locally compact version of the 
 Stone--Weierstrass theorem yields the second part. 
\end{proof}

Finally, we introduce a set that will be used as domain for linear operators: 
\begin{align}\label{eq:domain}
\mathcal{D}:=\operatorname{span}\{ P^D, F^D  \}.
\end{align}
Observe that we here do not specify the Banach space containing $\mathcal{D}$.

\section{Optimality conditions}\label{IIstool}

We now develop optimality conditions for cylindrical functions and polynomials of measure arguments, which we need in order to apply the  positive maximum principle on $M_+(E^\Delta)$. Our first result, Theorem~\ref{th:firstsecondorder}, extends the classical first and second order Karush--Kuhn--Tucker conditions for functions on $\mathbb{R}_+^m$ (see, e.g., \cite{B:97}).

Throughout the section we let $\Dcal$ be the set defined in \eqref{eq:domain}.
Note that we use Lemma~\ref{IIIlem1} and Lemma~\ref{Fcont} to extend the polynomials and cylindrical functions from $M_+(E)$ to $M_+(E^{\Delta})$.

\begin{theorem} \label{th:firstsecondorder}
Let $f\in \Dcal$ and $\nu^*$ satisfy $f(\nu^*)=\max_{\nu\in M_+(E^\Delta)}f(\nu)$. Then the following first and second order optimality conditions are satisfied.
\begin{enumerate}

\item\label{iti}
For all $\mu,\nu \in M(E^{\Delta})$ with $\supp(\mu) \subseteq \supp (\nu^*)$, it holds that
$$\langle \partial f(\nu^*), \mu  \rangle=0\qquad\text{and}\qquad \langle \partial f(\nu^*), \nu  \rangle\leq 0.$$
In particular,
$\partial_x f(\nu^*)\leq 0 $ for all  $x \in E^{\Delta}$ and
vanishes for all $x \in \operatorname{supp}(\nu^*)$.

\item\label{itii} For all $\mu \in M(E^{\Delta})$ with $\supp (\mu) \subseteq \supp(\nu^*)$, it holds that
\[
\langle \partial^2 f(\nu^*), \mu^2 \rangle \leq 0.
\]
In particular,   $\partial^2_{xx} f(\nu^*)\leq 0 $ for all $x \in \operatorname{supp}(\nu^*)$. 
\end{enumerate}
\end{theorem}
\begin{proof}\ref{iti}: Let $x\in \operatorname{supp}(\nu^*)$.
We shall first consider the following perturbation $\nu^*- \varepsilon_n \nu_n$ with $(\nu_n)_{n \in \mathbb{N}}$ a sequence of non-negative measures converging to $\delta_x$ and $\varepsilon_n > 0$ converging to $0$. To this end we need to construct $(\nu_n)_{n \in \mathbb{N}}$ and $(\varepsilon_n)_{n \in \mathbb{N}}$ such that $\nu^*- \varepsilon_n \nu_n$ is a non-negative measure for all $n \in \mathbb{N}$.

 For each $n\in \mathbb{N}$, let $A_n$ be the ball of radius $\frac{1}{n}$ centered at $x$, intersected with $\operatorname{supp}(\nu^*)$. Since $A_n \subset \operatorname{supp}(\nu^*)$, we have $\nu^*(A_n)>0$ and the measures $\nu_n := \nu^*(\fdot \cap A_n)/\nu^*(A_n)$ converge weakly to $\delta_x$ as $n\rightarrow \infty$. Let $\varepsilon_n\in(0,\nu^*(A_n))$ and note that  for all  $B\in \mathcal{B}(E^\Delta)$ 
$$\nu^*(B)-\varepsilon_n\frac{\nu^*(B\cap A_n)}{\nu^*(A_n)}\geq \frac{\nu^*(B\cap A_n)}{\nu^*(A_n)}(\nu^*(A_n)-\varepsilon_n)\geq 0, $$
where the last inequality follows from the fact that $\varepsilon_n\in(0,\nu^*(A_n))$. Hence $\nu^*-\varepsilon_n\nu_n$ is a non-negative measure. Moreover, using that $\nu_n$ is a probability measure, we can compute
\begin{equation}\label{eqn3}
\|(\langle g_1,\varepsilon_n\nu_n\rangle,\ldots,\langle g_k, \varepsilon_n\nu_n\rangle)\|^2\leq \varepsilon_n^2\sum_{i=1}^k\|g_i\|^2.
\end{equation}
By the maximality of $\nu^*$ and the Taylor expansion of $\phi$ we then have
\begin{equation*}
\begin{aligned} 0 &\geq f(\nu^*-\varepsilon_n\nu_n )-f(\nu^*)\\
&= \phi(\langle g_1, \nu^*-\varepsilon_n\nu_n \rangle,\ldots,\langle g_k, \nu^*-\varepsilon_n\nu_n \rangle)- \phi(\langle g_1, \nu^* \rangle,\ldots, \langle g_k,\nu^* \rangle)\\
&=-\varepsilon_n\sum_{i=1}^k \langle g_i, \nu_n\rangle\partial_i \phi(\langle g_1, \nu^*  \rangle,\ldots,\langle g_k, \nu^*  \rangle) + o(\varepsilon_n)\\
&=-\varepsilon_n \langle \partial f(\nu^*),\nu_n\rangle+ o(\varepsilon_n).
\end{aligned}
\end{equation*}
Dividing by $\varepsilon_n$, sending $n$ to infinity, and using the fact that $\nu_n$ converges weakly to $\delta_x$ we obtain 
$
\partial_x f(\nu^*)\geq 0
$
 for $x\in \operatorname{supp}(\nu^*)$. 
For general $x \in E^{\Delta}$, it holds that $\nu^* + \varepsilon \delta_x$ is a non-negative measure. By the maximimality of $\nu^*$ we thus have
\begin{equation*}
\begin{aligned} 0 &\geq f(\nu^*+\varepsilon \delta_x)- f(\nu^*) =\\
&= \phi(\langle g_1, \nu^* + \varepsilon \delta_x \rangle,\ldots, \langle g_k, \nu^* + \varepsilon \delta_x\rangle) - \phi(\langle g_1, \nu^* \rangle,\ldots,\langle g_k, \nu^* \rangle)\\
&=\varepsilon\sum_{i=1}^k g_i(x) \partial_i \phi(\langle g_1, \nu^*  \rangle,\cdots,\langle g_k, \nu^*  \rangle) + o(\varepsilon)\\
&=\varepsilon \partial_x f(\nu^*)+ o(\varepsilon).
\end{aligned}
\end{equation*}
Dividing again by $\varepsilon $ and sending it to $0$, we deduce that $\partial_x f \leq 0$ for all $x\in E^{\Delta}$.

Hence $\partial_x f = 0$ for all $x\in\operatorname{supp}(\nu^*)$  and $\partial_x f \leq 0$ for all $x\in E^{\Delta} \setminus \operatorname{supp}(\nu^*)$. This proves assertion~\ref{iti} as for all signed measures $\mu \in M(E^{\Delta})$ with $\supp(\mu) \subseteq \supp(\nu^*)$
\[
\langle \partial f(\nu^*), \mu \rangle= \int_{\supp(\mu)} \partial_x f(\nu^*)\mu(dx)=0.
\]

\ref{itii}: Let $\mu \in M_+(E^{\Delta})$ such that $\supp(\mu) \subseteq \supp (\nu^*)$.
Since $\nu^*$ is a maximum for $f$, we obtain from the second order Taylor expansion of $\phi$, the estimate~\eqref{eqn3}, and by using~\ref{iti}
\begin{equation}\label{eq:secondorder}
\begin{split}
 0 &\geq f(\nu^*+\varepsilon \mu)- f(\nu^*)= \\
&=\varepsilon \sum_{i=1}^k \partial_i \phi(\langle g_1, \nu^*  \rangle,\ldots,\langle g_k, \nu^*  \rangle) \langle g, \mu \rangle  \\
& \quad + \frac{1}{2}\varepsilon^2 \sum_{i,j=1}^n \partial_{ij} \phi(\langle g_1, \nu^*  \rangle,\ldots,\langle g_k, \nu^*  \rangle) \langle g_i \otimes g_j, \mu^2 \rangle + o(\varepsilon^2)\\
&=\frac{1}{2}\varepsilon^2 \langle \partial^2 f(\nu^*), \mu^2 \rangle+ o(\varepsilon^2).
\end{split}
\end{equation}
Therefore, $\langle \partial^2 f(\nu^*), \mu^2 \rangle \leq 0$ for $\mu\in M_+(E^{\Delta})$ with $\supp(\mu) \subseteq \supp (\nu^*)$. In particular, for $x\in\operatorname{supp}(\nu^*)$ and choosing $\mu = \delta_x$ we have
$$\partial^2_{xx} f(\nu^*)=\sum_{i,j=1}^n g_i(x) g_j(x)\partial_{ij} \phi(\langle g_1, \nu^*  \rangle,\ldots,\langle g_k, \nu^*  \rangle) \leq 0.
 $$ 
For the general signed measure case, define
\[
\mu_n:=\sum_{i=1}^m \lambda_i \nu_{n,i},
\]
for $m \in \mathbb{N}$, $\lambda_1, \ldots, \lambda_m \in \mathbb{R}$ and $\nu_{n,i}$ constructed as $\nu_n$ above  with $x$ replaced by $x_i \in \supp(\nu^*)$. Letting $\varepsilon_n$ decrease to $0$ sufficiently fast (note that only the negative $\lambda_i$ have to be taken into account) yields $\nu^*+ \varepsilon_n \mu_n \in M_+(E^{\Delta})$. The same reasoning as in \eqref{eq:secondorder} replacing $\varepsilon \mu$ by $\varepsilon_n \mu_n$ and passing to the limit implies that 
\[
\langle \partial^2 f(\nu^*), \Big(\sum_{i=1}^m \lambda_i \delta_{x_i}\Big)^{ 2} \rangle \leq 0.
\]
Passing to the weak closure yields $\langle \partial^2 f(\nu^*), \mu^2 \rangle \leq 0$ for all $\mu \in M(E^{\Delta})$ with $\supp(\mu) \subseteq \supp(\nu^*)$.
This implies the assertions of \ref{itii}.
\end{proof}

\begin{remark}
To compare the results of Theorem~\ref{th:firstsecondorder} with the classical conditions for functions on $\mathbb{R}^m_+$, let $E=\{1,\ldots, m\}$. Then $M_+(E)$ can be identified with $\mathbb{R}^m_+$. Condition~\ref{iti} translates to 
\[
\partial_i f(\nu^*)=0, \quad i \in E,
\]
if the $i$-th component of $\nu^*$ is strictly positive and $\partial_i f(\nu^*)\leq 0$ otherwise.
Condition~\ref{itii} corresponds to 
\[
y^{\top}\partial^2 f(\nu^*)y \leq 0
\]
for all $y \in \mathbb{R}^m$ such that $y_j=0$ if $\nu^*_j=0$. In particular
\[
\partial^2_{ii} f(\nu^*) \leq 0, \quad i \in E,
\]
 in the case where the $i$-th component of $\nu^*$ is strictly positive.
These  come from the classical Karush-Kuhn-Tucker conditions when the inequality constraints describe $\mathbb{R}_m^+$.
\end{remark}

Our next optimality condition is an slight adaptation of Theorem 3.4 in \cite{CLS:19}. In comparison with that result  we do not need to work with  a  group of positive isometries, but just with a  positive group. Let us recall here the definition of a positive group (see, e.g., \cite{arendt:86}).

\begin{definition}
A group $\{T_t\}_{t \in \mathbb{R}}$ on $C_{\Delta}(E)$ is positive if for
any $0 \leq g \in C_{\Delta}(E)$, we have 
$
T_t g  \geq 0.
$
\end{definition} 
We shall use the tensor notation $A\otimes A$  to denote the linear operator from $D\otimes D$ to $\widehat C_\Delta(E^2)$ determined by
$$(A\otimes A)(g\otimes g):=(Ag)\otimes (Ag)$$
 for a given linear operator $A\colon D\to  C_\Delta (E)$.

\begin{theorem}\label{IIL:KKT2}
Fix $f\in \Dcal$ and $\nu_*\in M_+(E^\Delta)$ with $f(\nu_*)=\max_{\nu \in M_+(E^\Delta)} f(\nu)$. Let $A$ be the generator of a strongly continuous positive group on $C_\Delta(E)$, and assume the domain of $A$ contains both $D$ and $A(D)$. Then
\[
\langle A^2(\partial f(\nu_*)), \mu\rangle + \langle (A\otimes A)(\partial^2 f(\nu_*)),\mu^2\rangle \le 0
\]
for every $\mu \in M_+(E^\Delta)$ with $\mu\le \nu_*$ and $\supp(\mu) \subseteq \supp(\nu_*)$.
\end{theorem}

\begin{proof}
In what follows we adapt the proof of Theorem 3.4 in \cite{CLS:19} to
the case $M_+(E^{\Delta})$.
Let $\{T_t\}_{t\in\R}$ be the group generated by $A$. For any $\mu\in M_+(E^\Delta)$, the group induces a flow of measures $\mu_t\in M(E^\Delta)$ via the formula $\langle g,\mu_t\rangle = \langle T_tg, \mu\rangle$ for $g\in C_\Delta(E)$. The positivity  implies that $\mu_t$ is non-negative. Indeed,  $\langle T_tg, \mu\rangle \geq 0$ for all $g \geq 0$, whence $\mu_t$ is a non-negative measure.\\

Therefore, assuming henceforth that $\mu\le\nu_*$, it follows that $\nu_*+\mu_t-\mu$ is a non-negative measure.
Since $\|T_t g-g\|=O(t)$ for every $g\in D$, we have $\langle g,(\mu_t-\mu)^k\rangle=O(t^k)$ for every $g\in D^{\otimes k}$. Maximality of $\nu_*$ and Taylor's formula then yield 
\begin{align}
0 &\ge f(\nu_*+\mu_t-\mu) - f(\nu_*)= \langle\partial f(\nu_*),\mu_t-\mu\rangle + \frac12 \langle\partial^2 f(\nu_*),(\mu_t-\mu)^2\rangle+o(t^2) \nonumber \\
&=\langle(T_t-\text{id})\partial f(\nu_*),\mu\rangle   + \frac12 \langle(T_t\otimes T_t-2\,T_t\otimes\text{id}+\text{id}\otimes\text{id})\partial^2f(\nu_*),\mu^2\rangle+o(t^2). \label{IIL:KKT2:eq1}
\end{align}
We claim that both $A$ and $-A$ satisfy the positive minimum principle (see Definition~\ref{def1}) on $E^\Delta$. 
Indeed, for $0 \leq g\in D$ and $x\in E^\Delta$ with $g(x)=\inf_{E^\Delta}g=0$, we have
\begin{equation}\label{eq:minboth}
\begin{split}
Ag(x)&=\lim_{t \to 0}\frac{T_tg(x)-g(x)}{t}=\lim_{t \to 0}\frac{T_tg(x)}{t} \geq 0, \\ -Ag(x)&=\lim_{t \to 0}\frac{T_{-t}g(x)-g(x)}{t}=\lim_{t \to 0}\frac{T_{-t}g(x)}{t} \geq 0; 
\end{split}
\end{equation}
which can just be seen as a consequence of \cite[Theorem B.II.16]{arendt:86}.
Since $-\partial f(\nu^*) \geq 0$ and $\partial_x f(\nu_*)=0$ for all $x\in\supp(\nu^*)$ due to Theorem \ref{th:firstsecondorder}, it follows that $A(\partial f(\nu_*))(x) = 0$ for all such $x$.
As a result, using that $\supp(\mu)\subseteq\supp(\nu_*)$ and that the domain of $A$ contains $A(D)$, we get
\begin{equation} \label{IIL:KKT2:eq2}
\langle(T_t-\text{id})\partial f(\nu_*),\mu\rangle = \langle(T_t-\text{id}-tA)\partial f(\nu_*),\mu\rangle = \frac12t^2\langle A^2( \partial f(\nu_*)),\mu\rangle+ o(t^2).
\end{equation}
Furthermore, using that
\[
(T_t\otimes T_t-2\,T_t\otimes\text{id}+\text{id}\otimes\text{id})(g\otimes g)=(T_tg-g)\otimes(T_tg-g)
\]
for all $g\in D$, we deduce that
\begin{equation} \label{IIL:KKT2:eq3}
\langle(T_t\otimes T_t-2\,T_t\otimes\text{id}+\text{id}\otimes\text{id})g,\mu^2\rangle = t^2 \langle (A\otimes A)g,\mu^2\rangle + o(t^2)
\end{equation}
for all $g\in D\otimes D$. Inserting \eqref{IIL:KKT2:eq2} and \eqref{IIL:KKT2:eq3} into \eqref{IIL:KKT2:eq1}, dividing by $t^2$, and sending $t$ to zero yields
\[
0 \ge \frac12 \langle A^2(\partial f(\nu_*)),\mu\rangle + \frac12 \langle (A\otimes A)\partial^2f(\nu_*),\mu^2\rangle.
\]
This completes the proof.
\end{proof}
The optimality condition obtained in Theorem~\ref{IIL:KKT2} involves the generator of a strongly continuous positive group and is more subtle than the previous ones. In particular in the finite-dimensional case the generator of a strongly continuous positive group can only be a diagonal matrix, see Lemma~\ref{lem:diagonal}. For the properties of such an operator in the general case, we refer to Remark~\ref{rem:genposgroup} and Appendix~\ref{app:C}.

 \begin{remark}\label{IIrem7}
We claim that for $A$ as in Theorem~\ref{IIL:KKT2}, the operator $A^2$ satisfies the positive minimum principle on $E^\Delta$. Indeed, let $g\in D$ and $x\in E^\Delta$ with $g(x)=\inf_{E^\Delta}g=0$. Then, as $A$ and $-A$ satisfy the positive minimum principle  $Ag(x)=0$. Hence $A^2g(x) = \lim_{t\downarrow0}(T_tg(x)-g(x)-Af(x))/t\geq 0$, which proves the claim.
\end{remark}

The following lemma illustrates how the conditions of  Theorem~\ref{IIL:KKT2} translate to the case where $D=C(E^{\Delta})$. This includes in particular the finite dimensional case when $E$ consists of $n \in \mathbb{N}$ points. 

\begin{lemma} \label{lem:diagonal}\phantomsection
\begin{enumerate}
\item\label{it2i} Let $A: C(E^{\Delta}) \to C(E^{\Delta}) $ be the generator of a strongly continuous positive group, then 
\[
Ag(x)=a(x)g(x)
\]
for some $a \in C_{\Delta}(E)$. Moreover $(A \otimes A)g(x,y)=a(x)a(y)g(x,y)$ for $g \in \widehat{C}_{\Delta}(E^2)$.

\item\label{it2ii}
Let $E$ consist of $n \in \mathbb{N}$ points 
and let $A$ be the generator of a strongly continuous positive group on $C_\Delta(E)\cong \mathbb{R}^n$. Then $A$ is a diagonal matrix.
\end{enumerate}
\end{lemma}

\begin{proof}
As shown in equation \eqref{eq:minboth}, $A$ and $-A$ satisfy the positive minimum principle. Moreover, as $D=C(E^{\Delta})$, Corollary~B.II.1.12 in \cite{arendt:86} yields that $A$ is bounded and that
\begin{equation}\label{eqn6}
-\| A \| g\leq Ag  \leq \| A \| g , \quad \text{for all } g \geq 0,
\end{equation}
where $\|A\|$ denotes the operator norm.
Observe that \eqref{eqn6}  yields
$$A((g-g(x))^+)(x)=A((g-g(x))^-)(x)= 0,$$
showing that $Ag(x)$ can only depend on $g(x)$. The statement  follows from the fact that $A$ is a bounded linear operator. The form of $(A \otimes A)g(x,y)$ is then obvious.

The second part  directly follows from the first one. Let us underline that a matrix in $\mathbb{R}^{n\times n}$ satisfies the positive minimum principle if and only if the off-diagonal elements are non-negative; see \cite[Example B.II.1.4]{arendt:86}. Since
both $A$ and $-A$ satisfy the positive minimum principle on $E$, we conclude that $A$ is a diagonal matrix. 
\end{proof}

\begin{remark}\label{rem:genposgroup}
Generators of strongly continuous positive groups on $C_{\Delta}(E)$ have been fully characterized in \cite[Theorem B.II.3.14]{arendt:86}. In  Appendix \ref{app:C} we recall the main tools behind their characterization by introducing definitions such as flows and cocycles.

\end{remark}

\section{Polynomial operators} \label{sec_pol_op}

Recall that $E$ is a locally compact Polish space and $
D \subseteq C_\Delta (E)
$
is a dense linear subspace containing the constant function 1. Set $\Dcal$ as in \eqref{eq:domain}. We now define polynomial operators, which constitute a class of possibly unbounded linear operators acting  on cylindrical functions and polynomials.  For the moment we  define these polynomial operators for polynomials on general subsets $ \s \subseteq M(E)$, similarly as in Section 4 of \cite{CLS:19}. To avoid unnecessary loss of generality we work with the set $\Mcal(M(E))$ of measurable real valued maps on $M(E)$.

\begin{remark}\label{rem:terminology}
In this section, and also more generally, we often work just with subsets of $M(E)$ despite the fact that the results hold also for subsets of $M(E^{\Delta})$. This is due to the fact that the second can be deduced from the first by means of a direct argument. Whenever this is not the case, we will explicitly state it.  
\end{remark}

\begin{definition}\label{IIdef1}
Fix $\s\subseteq M(E)$. A linear operator  $L:\Dcal\to \Mcal(M(E))$ is called $\s$-polynomial if
$L$ maps $P^D$ to  $P$  such that for every $p\in P^D$ there is some $q\in P$ with $q|_{\s}=Lp|_{\s}$ and
$\deg(q) \le \deg(p).$
\end{definition}

\subsection{Characterization of polynomial operators in the diffusion setting}

For a linear operator $L$,  the associated \emph{carr\'e-du-champ operator} (see, e.g., \cite{BGL:13}) is the symmetric bilinear map $\Gamma:\mathcal{D} \times \mathcal{D}\to \Mcal(M(E))$ 
\begin{equation} \label{IIeq:Gamma}
\Gamma(p,q)=L(pq)-pLq-qLp.
\end{equation}

Related to the carr\'e-du-champ operator is the notion of a \emph{derivation}, which we recall in the following definition and which can be used to characterize the form of $L$, as a second-order differential operator, typical for diffusion processes.

\begin{definition}\label{def:derivation}
Fix $\s\subseteq M(E)$. A symmetric bilinear map $\Gamma:\mathcal{D} \times \mathcal{D}\to \Mcal(M(E))$ is called an $\s$-derivation if for all $p,q,r\in \mathcal{D}$, $\Gamma(pq,r)=p\Gamma(q,r)+q\Gamma(p,r)$ on $\s$.
\end{definition}

\begin{remark}
As explained in Remark \ref{rem:terminology}
we consider  only subsets of $M(E)$ and not of $M(E^{\Delta})$. Hence for finite measure we shall only speak of $M_+(E)$-derivations (even if $\mathcal{S}=M_+(E^{\Delta})$).
\end{remark}

\begin{lemma}\label{lem2}
The carr\'e-du-champs operator $\Gamma$ of a linear operator $L:\mathcal{D}\to \Mcal(M(E))$ is an $M_+(E)$-derivation if and only if
$Lf(\nu)=0$ for each $f$ such that $\partial f(\nu)=0$ and $\partial^2f(\nu)=0$.
\end{lemma}
\begin{proof}
Observe first that if
$\Gamma$ is an $M_+(E)$-derivation, it holds  $\Gamma(1, 1) = 2\Gamma(1, 1)$, hence $L(1)=-\Gamma(1, 1) = 0$. Since for each $p,q,r\in \Dcal$ and $\nu\in M_+(E)$ it holds
\begin{align*}
   \Gamma(pq,r)(\nu)&-p(\nu)\Gamma(q,r)(\nu)-q(\nu)\Gamma(p,r)(\nu)-p(\nu)q(\nu)r(\nu)L1(\nu)\\
   &=L\Big((p-p(\nu))(q-q(\nu))(r-r(\nu))\Big)(\nu),
\end{align*}
we get that $\Gamma$ is a derivation if and only if $L1=0$ and
\begin{equation}\label{eqn5}
L\Big((p-p(\nu))(q-q(\nu))(r-r(\nu))\Big)(\nu)=0,
\end{equation}
for each $p,q,r\in \Dcal$ and $\nu\in M_+(E)$. Observe now that each $p\in \Dcal$ admits the representation $$p(\nu)=\phi(\langle g_1,\nu\rangle, \ldots, \langle g_n,\nu\rangle)$$
for some $\phi\in C^\infty(\R^n)$ and some linearly independent $g_1,\ldots,g_n\in D$. Assuming that $\partial p(\nu)= 0$ and $\partial^2p(\nu)= 0$ and applying Taylor's theorem to $\phi$ yields
$$p(\mu)=\sum_{k_1,k_2=1}^n\langle g_{k_1},\mu-\nu\rangle\langle g_{k_2},\mu-\nu\rangle\tilde p_{k_1,k_2}(\mu)$$
for some $\tilde p_{k_1,k_2}\in \Dcal$ such that $\tilde p_{k_1,k_2}(\nu)=0$. The claim then follows by \eqref{eqn5}.
\end{proof}

For a finite-dimensional diffusion it is known that its generator is polynomial if and only if the drift and diffusion coefficients are polynomial of first and second degree, respectively,  see~\cite{CKT:12} and \cite{FL:16}. The following result is the generalization of this fact to the measure-valued setting. 

\begin{theorem}\label{Lpol}
A linear operator $L:\mathcal{D}\to \Mcal(M(E))$ is $M_+(E)$-polynomial and its carr\'e-du-champ operator $\Gamma$ is an $M_+(E)$-derivation if and only if $L$ admits a representation
\begin{align}
Lf(\nu) = &B_0(\partial f(\nu)) + \langle B_1(\partial f(\nu)),\nu\rangle\label{eq:op}\\
& + \frac{1}{2} \Big(Q_0(\partial^2 f(\nu)) + \langle Q_1(\partial^2 f(\nu)), \nu\rangle + \langle Q_2(\partial^2 f(\nu)), \nu^2\rangle\Big), \quad f\in \Dcal,  \nu\in M_+(E)\notag
\end{align}
for some linear operators $B_0:D\to\R$, $B_1:D\to  C_\Delta(E)$, $Q_0:D\otimes D\to\R$, $Q_1:D\otimes D\to C_\Delta(E)$, $Q_2:D\otimes D\to \widehat C_\Delta(E^2)$. These operators are uniquely determined by $L$.
\end{theorem}
\begin{proof}
Theorem A.1 in \cite{CLS:19} yields the result when $\Dcal$ is replaced by $P^D$.  The extension of the first implication from $P^D$ to $\Dcal$ is then direct. For the converse implication, 
observe that for each $f\in \Dcal$ and $\nu\in M_+(E)$ there is a $p\in P^D$ such that $\partial (p-f)(\nu)=0$ and $\partial^2(p-f)(\nu)=0$. The claim then follows by Lemma~\ref{lem2}.
\end{proof}

\subsection{Dual operators}

 We would now like to associate to an $\Scal$-polynomial operator $L$ a family of so-called \emph{dual} operators  $(L_k)_{k \in \mathbb{N}}$ which are linear operators mapping the coefficients vector of $p$ to the coefficients vector of $Lp$. To this end we recall Definition 2.3 of \cite{CS:19}.

\begin{definition} \label{IIeq:L on Dk}
Fix  $m\in \N_0$ and let $L$ be an $\Scal$-polynomial operator. 
An \emph{$m$-th dual operator} corresponding to $L$ is a linear operator  $L_m: \bigoplus_{k=0}^m D^{\otimes k}\to
\bigoplus_{k=0}^m \widehat{C}_\Delta (E^k)$
such that $L_m\vec g =:(L_m^0\vec{g},\ldots, L_m^m\vec{g})$ satisfies
$$
Lp(\nu)= \sum_{k=0}^m \langle L^k_m \vec{g},\nu^k\rangle
\qquad \text{for all } \nu \in \Scal,
$$
where $p(\nu) = \sum_{k=0}^m \langle g_k,\nu^k\rangle$.
 Whenever $L_m$ is a closable operator\footnote{We refer to \cite[Chapter 1]{EK:09} for the precise definition.}, we still denote its closure  by $L_m: \mathcal{D}(L_m)\to \bigoplus_{k=0}^m \widehat{C}_\Delta (E^k)$ and its domain by $\mathcal{D}(L_m)\subseteq \bigoplus_{k=0}^m \widehat{C}_\Delta (E^k)$. 
\end{definition}

If $\mathcal{S}=M_+(E)$, then the dual operator is unique.
Indeed, in this case
the representation \eqref{IIeq:p(nu)} is unique, which is a consequence of Corollary~2.4 in \cite{CLS:19}, and we
can thus identify each polynomial with its coefficients vector. We therefore call it \emph{the} dual operator.
It is interesting to note that if $L$ satisfies \eqref{eq:op}, then 
\begin{align*}
L_m^m(g^{\otimes m})&=
mB_1(g)\otimes g^{\otimes(m-1)}+\frac {m(m-1)}2 Q_2(g\otimes g)\otimes g^{\otimes (m-2)},\\
L_m^{m-1}(g^{\otimes m})&=mB_0(g) g^{\otimes(m-1)}+\frac {m(m-1)}2 Q_1(g\otimes g)\otimes g^{\otimes (m-2)},
\end{align*}
$L_m^{m-2}(g^{\otimes m})=\frac {m(m-1)}2 Q_0(g\otimes g)g^{\otimes (m-2)}$, and $L_m^k(g^{\otimes m})=0$ for each $k< m-2$.

\section{Moment formula and existence of polynomial diffusions on $M_+(E)$} \label{sec_ex}

Let $L$ be a linear operator acting on $\mathcal{D}$ for $\Dcal$ specified  in \eqref{eq:domain}. In this section we study existence and uniqueness of non-negative measure valued polynomial diffusions which we introduce via the martingale problem.
In the following let $\mathcal{S}$ stand either for $M_+^{\mathfrak{\Delta}}(E^{\Delta})$, $M_+(E^{\Delta})$ or $M_+(E)$.

An $\mathcal{S}$-valued process $X$ with c\`adl\`ag paths defined on some filtered probability space $(\Omega, \Fcal,(\Fcal_t)_{t\geq0}, \P)$ is called an $\mathcal{S}$-valued {\em solution to the martingale problem for $L$} with initial condition $\nu\in \mathcal{S}$ if $X_0=\nu$ $\P$-a.s.~and
\begin{equation}\label{martprob}
N^f_t = f(X_t) - f(X_0) - \int_0^t Lf(X_s) ds
\end{equation}
defines a local martingale for every $f$ in the domain of $L$.
Uniqueness of solutions to the martingale problem is always understood in the sense of law. The martingale problem for $L$ is {\em well--posed} if for every initial condition $\nu \in \mathcal{S}$ there exists a unique $\mathcal{S}$-valued solution to the martingale problem for $L$ with initial condition~$\nu$.

We are mainly interested in $M_+(E)$-valued solutions with continuous paths (with respect to the topology of weak convergence) corresponding to polynomial operators.

\begin{definition}
Let $L$ be $M_+(E)$-polynomial. Any continuous $M_+(E)$-valued solution to the martingale problem for $L$ is called a
\emph{measure-valued polynomial diffusion}.
\end{definition}

The following lemma relates path continuity of $M_+(E)$-and $M_+(E^{\Delta})$-valued solutions to the martingale problem to the carr\'e-du-champ operator being a derivation. This explains why we consider derivations in Theorem~\ref{Lpol}.

\begin{lemma}\label{IIlem5}
If the carr\'e-du-champ operator $\Gamma$ associated to $L$ is an $M_+(E)$-derivation, then any $M_+(E)$- or $M_+(E^{\Delta})$-valued solution to the martingale problem for $L$ has continuous paths and  the corresponding quadratic covariation structure is given by
$$d[N^f,N^g]_t=2\Gamma(f,g)(X_t)dt,$$
for every $f,g\in \Dcal$.
 
Conversely, if for every initial condition $\nu\in M_+(E)$ (or $M_+(E^{\Delta})$ respectively) there is an $M_+(E)$- (or $M_+(E^{\Delta})$ respectively) valued solution to the martingale problem for $L$ with continuous paths, then the carr\'e-du-champ operator $\Gamma$ associated to $L$ is an $M_+(E)$-derivation.
\end{lemma}

\begin{proof}
We here state the proof only for $M_+(E)$ since the $M_+(E^{\Delta})$-case follows by the same arguments extending $f \in \mathcal{D}$ to $M(E^{\Delta})$.

Let $X$ be an $M_+(E)$-valued solution to the martingale problem for $L$. By Proposition~2 in \cite{BE:85}, the real-valued process $f(X)$ is continuous for every $f\in \mathcal{D}$, in particular for every linear monomial $f(\nu)=\langle g,\nu\rangle$ with $g\in D$. Since $D$ is dense in $C_\Delta (E)$, we can conclude that $X$ is continuous with respect to the topology of weak convergence on $M_+(E)$.

Conversely, if $X$ is a $M_+(E)$-valued solution to the martingale problem for $L$ with continuous paths, then, by Lemma \ref{IIL:Psmooth} and Lemma \ref{Fcont}, the map $t\mapsto f(X_t)$ is continuous for all $f\in \mathcal{D}$. The result now follows by Proposition~1 in \cite{BE:85}.
\end{proof}

\subsection{Moment formula}\label{sec_moments_uniqueness}

Polynomial diffusions are of particular interest in many applications because they satisfy a {\em moment formula}, which allows to compute the process' moments in a tractable way. If $E$ is a finite set, the moment formula always holds, but technical conditions, in particular on the dual operators, are needed in the general case. 
We apply here Theorem 3.4 and Remark 3.20 of \cite{CS:19} (see also Example 3.21 and Lemma 3.2 in \cite{CS:19}) to get Theorem~\ref{thm1}.

Let $L$ be an $M_+(E)$-polynomial operator, fix  $m\in \N_0$, and let $L_m$ be the closable $m$-th dual operator corresponding to $L$ with domain $\mathcal{D}(L_m)$.
Before stating the theorem,  we extend the domain of $L$ to polynomials with coefficients in $\mathcal{D}(L_m)$ by setting
\begin{equation}\label{eqn21}
p_{\vec{g}}=\sum_{k=0}^m \langle g_k,\nu^k\rangle \quad \text{and} \quad Lp_{\vec{g}}=\sum_{k=0}^m \langle L^k_m \vec{g},\nu^k\rangle
\end{equation}
for all $\vec{g} \in \mathcal{D}(L_m)$ and $m\in \N_0$. 
As in the finite dimensional case the moment formula corresponds to a solution of a system of linear ODEs. Note that we shall usually call these differential equations ODEs, even though they correspond in many cases to PDEs or PIDEs depending on the underlying space and the specification of the operators.  In the current infinite dimensional setting we also need to make the solution concept precise. The following is an adaptation of Definition 3.3 in \cite{CS:19} to the current measure-valued setting.

\begin{definition}\label{def:sol}

We call  a function $t \mapsto \vec{g}_t$ 
with values in $\mathcal{D}(L_m)$  a solution of the $m+1$ dimensional system of ODEs
\[
\partial_t \vec{g}_t =  L_m \vec{g}_t , \quad \vec{g}_0=g,		
\]
if for every $t >0$ it holds
\begin{align}\label{eq:ODE}
\sum_{k=0}^m \langle g_{t,k},\nu^k\rangle=\sum_{k=0}^m \langle g_{0,k},\nu^k\rangle +\int_0^t \sum_{k=0}^m \langle L^k_m \vec{g}_s, \nu^k \rangle ds
\end{align}
for all $\nu \in M_+(E^{\Delta})$.
\end{definition}

\begin{remark}
Note that the above solution concept reduces to a more classical one if we take $\nu=\delta_{x_1} + \cdots +\delta_{x_k}$ with $x_i \in E$, $i=1, \ldots, k$ and $k =1, \ldots, m$. Indeed, by polarization \eqref{eq:ODE} can be transformed into
\begin{align*}
g_{t,0}&=g_{0,0}+ \int_0^t L_m^0\vec{g}_s ds\\
g_{t,1}(x_1)&=g_{0,1}(x_1)+ \int_0^t L_m^1\vec{g}_s(x_1) ds\\
&\vdots\\
g_{t,m}(x_1, \ldots, x_m)&=  g_{0,m}(x_1, \ldots, x_m) +\int_0^t L^m_m \vec{g}_s(x_1, \ldots, x_m) ds
\end{align*}
and thus reduces to a classical (except of the integral form) solution of a multivariate P(I)DE. 
\end{remark}

\begin{theorem}[Dual moment formula]\label{thm1}
Let $L$ be an $M_+(E)$-polynomial operator, fix  $m\in \N_0$, let $L_m$ be the  $m$-th dual operator corresponding to $L$, and assume that $L_m$ is closable with domain $\mathcal{D}(L_m)$. Suppose that an $M_+(E^{\Delta})$-valued solution $(X_t)_{t \geq 0} $ to the martingale problem for $L$ exists.
Fix a coefficients vector $\vec{g}=(g_0, \ldots, g_m)\in\mathcal{D}(L_m)$ and  suppose that the following condition holds true.
\begin{itemize}
\item 
There is a solution  in the sense of Definition \ref{def:sol} of the $m+1$ dimensional system  of linear ODEs on $[0,T]$ given by
\begin{equation}\label{ODE}
\partial_t \vec{g}_t =  L_m \vec{g}_t, \qquad 
\vec{g}_0=\vec{g}.					
\end{equation}
\end{itemize}
Then, for all $0\leq t\leq T$ the representation
\begin{equation*}
\mathbb{E}\left[  \sum_{k=0}^m \langle g_k, X_T^k \rangle \,\big|\, \mathcal{F}_t \right ]= \sum_{k=0}^m \langle g_{T-t,k}, X_t^k \rangle
\end{equation*}
 holds almost surely. 
\end{theorem}

\begin{proof}
This is simply a consequence of Theorem 3.4 in \cite{CS:19} and Lemma \ref{lem:martingality} below.
\end{proof}

\begin{example}\label{ex1}
Let $L$ be an $M_+(E)$-polynomial operator, fix  $m\in \N_0$, and let $L_m$ be the closable $m$-th dual operator corresponding to $L$. Assume that $L_m^0=0,\ldots, L_m^{m-1}=0$ and that $L_m^m|_{D^{\otimes m}}$ satisfies the positive minimum principle. If the closure $\overline L_m^m$ of $L_m^m$ generates a strongly continuous positive semigroup $(Y^{k}_t)_{t\geq0}$, then $\vec g_t:=(Y_t^kg) \vec e_k$ satisfies \eqref{ODE} for every $g$ in the domain of $\overline L_m^m|_{D^{\otimes m}}$. We provide now a simple example to illustrate this mechanism.

Consider the process given by $X_t=S_t\mu$ where $dS_t=\sigma S_tdW_t$, $S_0=1$, and $\mu\in M_+(E)$. Since $\langle g, X\rangle$ is a local martingale and
$$
\langle g, X_t\rangle^2=(S_t\langle g,\mu\rangle )^2
=\frac 1 2 \int_0^t\langle \sigma^22g\otimes g,X_s^2\rangle ds +\text{(local martingale)},
$$
the generator of $X$ is given by
$Lp(\nu)=\frac 1 2 \langle \sigma^2 \partial^2p(\nu),\nu^2\rangle.$
Observe that
$$L(\langle g, \,\cdot\,\rangle^n)(\nu)=\langle L_n^n g^{\otimes n},\nu^n\rangle,$$
where $L_n^n g^{\otimes n}=\sigma^2\frac {n(n-1)} 2g^{\otimes n}$ and note that $L_n^n$ is the generator of the semigroup given by
$$Y_tg^{\otimes n}(z):=\E[g^{\otimes n}(Z^{(n)}_t)\exp(\int_0^t m(Z^{(n)}_s)ds)|Z^{(n)}_0=z],$$
where $Z^{(n)}$ is the constant process (hence the process generated by 0) and $m(z)=\sigma^2\frac {n(n-1)} 2$. This in particular yields
$$\langle Y_tg^{\otimes n},X_0^n\rangle=\langle g, \mu\rangle^n \exp({\sigma^2 t}\frac{n(n-1)}2)
=\E[\langle g,\mu\rangle^n S_t^n]=\E[\langle g, X_t\rangle^n].$$
\end{example}

\begin{lemma}\label{lem:martingality}
Let $L$ be an $M_+(E)$-polynomial operator and suppose that an $M_+(E^{\Delta})$-valued solution $(X_t)_{t \geq 0 }$
to the martingale problem for $L$ exists.
\begin{enumerate}
\item\label{it3i}
Then, for every $p \in P^D$, 
\[
\mathbb{E}[p(X_t)^2] < Ce^{Ct},
\]
for some $C>0$.
Moreover, 
$N^p$ as defined in \eqref{martprob} is a true martingale. 

\item\label{it3ii} Moreover,  if there is a solution to \eqref{ODE} in the sense of Definition \ref{def:sol}, then 
Condition (ii) and (iii) of Theorem 3.4 in \cite{CS:19} also hold true.
\end{enumerate}
\end{lemma}

\begin{proof}
Let $\| \fdot \|_{M(E^{\Delta})}$ denote the total variation norm on ${M}(E^{\Delta})$. Then for all $\nu \in M_+(E^{\Delta})$, it holds $\| \nu \|_{M(E^{\Delta})} = \nu(E^{\Delta})$
and hence
\[
1 + \| \nu\|_{M(E^{\Delta})}^{2m} =1+ \nu(E^{\Delta})^{2m}, \quad \nu \in M_+(E^{\Delta}).
\]
Thus, for any $p=\sum_{k=0}^m \langle g_k, \nu^k \rangle \in P$, we can estimate 
\[
p^2  \leq C(1 + \nu(E^{\Delta})^{2m})
\]
 and  $$|L(1+  \nu(E^{\Delta})^{2m})| \leq C(1+  \nu(E^{\Delta})^{2m})$$ for some constant $C >0$. 
According to Definition 3.18 in \cite{CS:19}, this means that every $p$ of degree $m$ is $(C,1+  \nu(E^{\Delta})^{2m})$ bounded.  Lemma 3.19 in \cite{CS:19} thus yields that for every  $p \in P$, 
\[
\mathbb{E}[p(X_t)^2] < Ce^{Ct}
\]
and the local martingale $N^p$ is actually a true (even square integrable) martingale. 

Concerning part \ref{it3ii}, Lemma 3.19 in \cite{CS:19} also yields Condition (ii) and (iii) of Theorem~3.4 as long as we have solution to \eqref{ODE}.
\end{proof}

\subsection{Existence}\label{IIs40}

Our first main result of this section gives  sufficient conditions for the existence of solutions to the martingale problem. Applications of this result are discussed in Section \ref{sec:applications}. 
Recall that $E$ is throughout a locally compact Polish space. We start by a definition which will be used to describe the form of $Q_2$ in the representation of $L$ given by \eqref{eq:op}.

\begin{definition}\label{def2}
We say that a linear operator $C$ admits a  $(\beta, \pi)$-representation if
\begin{align}\label{eq:copositiveconcrete}
C(g)(x,y)=\frac{1}{2}(\pi(x,y) g(x,x) + \pi(y,x) g(y,y) +2 \beta(x,y) g(x,y)), \quad  g \in D \otimes D
\end{align}
where
\begin{itemize}
\item  $\beta: (E^{\Delta})^2 \to \mathbb{R}$ is a symmetric function such that $\beta(x,x) \geq 0$ for all $x \in E^{\Delta}$;
\item $\pi: (E^{\Delta})^2 \to \mathbb{R}_+$ is a non-negative function  such that $\pi(x,x)=0$ for all $x \in E^{\Delta}$;
\item for all $n\in \mathbb{N}$, $x_1, \ldots, x_n \in E^{\Delta}$, $c_1, \ldots, c_n \in \mathbb{R}_{++}$, the $n \times n$ matrix 
\begin{align}\label{eq:matrix}
A^{(n)}:=\beta_n+ \begin{pmatrix} \sum_{j=1}^n \frac{c_j}{c_1} \pi(x_1, x_j) &  &\\
& \ddots &\\
& & \sum_{j=1}^n \frac{c_j}{c_n} \pi(x_n, x_j)
 \end{pmatrix} \in \mathbb{S}^n_+,
\end{align}
where $\beta_n \in \mathbb{S}^n$ with entries $\beta_{n,ij}=\beta(x_i,x_j)$;
\item the map $ (x,y) \mapsto \frac{1}{2}(\pi(x,y) g(x,x) + \pi(y,x) g(y,y)+ 2\beta(x,y)g(x,y))$ lies in $\widehat{C}_{\Delta}(E^2)$ for all $g \in D\otimes D$. 

\end{itemize}
\end{definition}

Accordingly, our sufficient conditions for the  existence of the martingale problem now read as follows.

\begin{theorem}\label{main1}
Suppose that for $i=1, \ldots n$, each $A_i$ is the generator of a strongly continuous positive group on $C_\Delta(E)$ such that its domain contains both $D$ and $A_i(D)$.
Let $L:\Dcal\to C(M_+(E))$ be a linear operator of form~\eqref{eq:op}, where
\begin{enumerate}
\item\label{it4iv} $B_0: D \to \mathbb{R}$ is given by $B_0(g) =\langle g,b\rangle$ with $b \in M_+(E^{\Delta})$;
\item\label{it4v} $B_1- \frac{1}{2}\sum_{i=1}^n A_i^2: D \to C_{\Delta}(E)$ satisfies the positive minimum principle on $E^{\Delta}$;
\item\label{it4i} $Q_0 \equiv 0$;
\item\label{it4ii} $Q_1$ is of the form 
\[
Q_1(g)= \alpha \diag(g), \quad \text{ that is,} \quad  Q_1(g)(x)=\alpha(x)g(x,x), \quad g \in D \otimes D,
\]
where $\alpha \in C_{\Delta}(E)$ with values in $\mathbb{R}_+$;
\item\label{it4iii} $Q_2$ is of the form 
\[
Q_2(g)= C(g)+ \sum_{i=1}^n (A_i \otimes A_i)(g), \quad g \in D \otimes D,
\]
where $C$ admits a ($\beta, \pi$)-representation.
\end{enumerate}
Then $L$ is $M_+(E)$-polynomial and its martingale problem has an $M_+(E^{\Delta})$-valued solution with continuous paths for every initial condition $\nu\in M_+(E^{\Delta})$.

If additionally, the measure $b$ and the initial condition $\nu$ both lie in $ M_+(E)$
and if there exists some function $m \in C_{\Delta}(E)$ such that
 $$\text{bp-}\lim_{n \to \infty}(B_1(1-g_n)-m(1-g_n))\leq 0,$$
   for some sequence 
   $(g_n)_{n\in \N} \in D \cap C_0(E)$ satisfying $\text{bp-}\lim_{n \to \infty}g_n= 1_E$,
then any solution to the martingale problem takes values  in $M_+(E)$. 
\end{theorem}

\begin{remark}
\begin{enumerate}
\item
When the state space are probability measures $M_1(E)$, Theorem 5.6 in \cite{CLS:19} provides sufficient conditions guaranteeing existence of solutions to the corresponding martingale problem. These conditions are stronger and imply the current ones. Indeed, we identify $B_1$ and $Q_2$ here with $B$ and $Q$ there. More specifically, $C(g)$ is identified with $\alpha \Psi(g)$ and $A_i$
with $A_i$. All other quantities are set to $0$.
The conditions on $B$ and $A_i$ are clearly stronger. In particular, the positive maximum principle implies the positive minimum principle on $E^{\Delta}$. Moreover, $\alpha \Psi$ admits a $(\beta,\pi)$-representation for 
$\pi(x,y)=\pi(y,x)=-\beta(x,y)=\alpha(x,y)1_{\{x\neq y\}}$. Observe indeed that in this case the matrix $A^{(n)}$ appearing in \eqref{eq:matrix} is given by $$A^{(n)}_{ij}:=-\alpha(x_i,x_j)1_{\{x_i\neq x_j\}}+\sum_{k=1}^n\frac {c_k}{c_i}\alpha(x_i,x_k)1_{\{x_i\neq x_k\}}1_{\{i=j\}},$$
which is positive semidefinite since
$
\sum_{ij=1}^n A^{(n)}_{ij}v_iv_j=\sum_{i=1}^n\alpha(x_i,x_j)\frac {c_j}{c_j}\big(v_i-v_j\frac{c_j}{c_i}\big)^2\geq0$
for each $v_i\in \R$.

\item
Note that \eqref{eq:op} imposes the implicit condition   that $C(g) $ has to lie in $\widehat C_\Delta(E^2)$ for every $g\in D\otimes D$. If $D=C_\Delta(E)$, then this necessarily yields boundedness conditions on $\pi$ and $\beta$, as is seen from Theorem~\ref{mainthm} below. However, this does not hold for general $D\subseteq C_\Delta(E)$, as one can see by considering $E=\R$, $D\subseteq C^1_\Delta(\R)$, and $\pi(x,y)=\pi(y,x)=-\beta(x,y)=|x-y|^{-1}1_{\{x\neq y\}}$. 
\end{enumerate}
\end{remark}

\begin{proof}[Proof of Theorem~\ref{main1}]
Theorem~\ref{Lpol} implies that $L$ is $M_+(E)$-polynomial. Lemma~\ref{IIIlem8}\ref{it7i} yields the existence of an $M_+(E^{\Delta})$-valued  continuous solution to the martingale problem for any initial condition provided that $L$ satisfies the positive maximum principle on $M_+(E^\Delta)$, which we shall check below.

Let us now verify that the positive maximum principle on $M_+(E^\Delta)$ holds true.
Let therefore $\nu^*\in M_+(E^\Delta)$ be a maximizer of $f \in \mathcal{D}$ over $M_+(E^\Delta)$. Then the optimality conditions in Theorem \ref{th:firstsecondorder} yield
\begin{equation}\label{eq:optcondapplied}
\begin{split}
\partial_x f (\nu^*)\leq 0, \quad  \langle \partial f (\nu^*), \nu^* \rangle =0,\quad
\partial^2_{xx} f (\nu^*)\leq 0  \quad \text{and} \quad \langle \partial^2 f(\nu^*), \mu^2 \rangle \leq 0,
\end{split}
\end{equation}
for each $x\in E$ and $\mu \in M(E^{\Delta})$ with $\supp(\mu) \subseteq \supp(\nu^*)$.
We now analyze $Lf(\nu^*)$, which reads due
the conditions \ref{it4iv} - \ref{it4iii} as follows:
\begin{align*}
Lf(\nu^*)&= \langle  \partial f(\nu^*),b \rangle + \langle B_1(\partial f(\nu^*)), \nu^* \rangle +\frac{1}{2}\langle \alpha \diag(\partial^2 f(\nu^*)), \nu^* \rangle\\
&\quad +  \frac{1}{2}\langle C(\partial^2 f(\nu^*)), (\nu^*)^2 \rangle + \frac{1}{2}  \langle\sum_{i=1}^n (A_i \otimes A_i)(\partial^2 f(\nu^*)),  (\nu^*)^2 \rangle. 
\end{align*}
By \eqref{eq:optcondapplied} and the positivity of $b$ and $\alpha$,  we have
\begin{equation}\label{eq:negative}
\langle  \partial f(\nu^*),b \rangle \leq 0\qquad\text{ and }\qquad
\langle \alpha \diag(\partial^2 f(\nu^*)), \nu^* \rangle \leq 0.
\end{equation}
Next, observe that choosing the signed measure in \eqref{eq:optcondapplied} to be $\sum_{i=1}^n \lambda_i\delta_{x_i}$ for some $\lambda_i \in \mathbb{R}$ yields
$\sum_{i,j=1}^n \lambda_i\lambda_j \partial^2_{x_ix_j}f(\nu^*)\leq 0$,
which implies that the matrix $(H_{n,ij})_{ij=1}^n$ given by $H_{n,ij}:=\partial^2_{x_ix_j}f(\nu^*)$
is negative semidefinite. Taking $\mu\in M_+(E^\Delta)$ such that  $\supp(\mu)$ consists of $n$ points
$x_1, \ldots, x_n \in \supp(\nu^*)$, we thus have $\mu=\sum_{i=1}^n c_i \delta_{x_i}$ for some $c_1, \ldots, c_n >0$.

Since $C(g)$ admits a ($\beta, \pi$)-representation,  
using the notation of Definition~\ref{def2} we can thus write
$$
\langle C(\partial^2 f(\nu^*)),\mu^2 \rangle 
=\sum_{i,j=1}^n(A^{(n)}\circ H_n)_{ij}c_ic_j,
$$
where $\circ$ denotes the Hadamard product, i.e.~the componentwise multiplication. As $A^{(n)}$ is by assumption positive semidefinite and as the Hadmard product between a positive semidefinite and negative semidefinite matrix is negative definite, we can conclude that
\begin{align}\label{eq:C}
\langle C(\partial^2 f(\nu^*)), \mu^2 \rangle \leq 0.
\end{align}
Passing to the weak closure yields \eqref{eq:C}.

Finally, observe that from \eqref{eq:negative} and \eqref{eq:C} we get that
\[
Lf(\nu^*) \leq \langle B_1(\partial f(\nu^*)), \nu^* \rangle + \frac{1}{2}  \langle\sum_{i=1}^n (A_i \otimes A_i)(\partial^2 f(\nu^*)),  (\nu^*)^2 \rangle.
\]
As $B_1- \frac{1}{2}\sum_{i=1}^n A_i^2$ satisfies the positive minimum principle, we have by \eqref{eq:optcondapplied}
\[
\langle B_1(\partial f(\nu^*))- \frac{1}{2}\sum_{i=1}^n A_i^2(\partial f(\nu^*)), \nu^*  \rangle \leq 0.
\]
Therefore,
\[
Lf(\nu^*) \leq \frac{1}{2}\langle \sum_{i=1}^n A_i^2(\partial f(\nu^*)), \nu^* \rangle +\frac{1}{2}  \langle\sum_{i=1}^n (A_i \otimes A_i)(\partial^2 f(\nu^*)),  (\nu^*)^2 \rangle \leq 0,
\]
where the last inequality follows from Theorem~\ref{IIL:KKT2}. This proves the positive maximum principle on $M_+(E^{\Delta})$ and thus the existence of an $M_+(E^{\Delta})$-valued solution. The remaining statement on $M_+(E)$-valued solutions to the martingale problem follows by Lemma~\ref{IIIlem8}\ref{it7ii}.
\end{proof}

Even though Theorem \ref{main1} only gives sufficient conditions for the existence of solutions to the martingale problem for general coefficient domains $D$, the next result shows that the conditions are sharp.
Indeed, we consider the case $D=C_\Delta(E)$ and get the following characterization.

\begin{theorem}\label{mainthm}
Let $D=C_\Delta(E)$ and let $L:\Dcal\to C(M_+(E))$ be a linear operator.
Then $L$ is $M_+(E)$-polynomial, there exists an $M_+(E^{\Delta})$-valued solution to the martingale problem for all initial conditions in $M_+(E^\Delta)$ and all solutions have continuous paths, if and only if $L$ satisfies \eqref{eq:op} with
\begin{enumerate}
\item\label{it8iv} $B_0: C_{\Delta}(E) \to \mathbb{R}$ is given by $B_0(g) =\langle  g,b\rangle$ with a Radon measure $b \in M_+(E^{\Delta})$;
\item\label{it8v} For every $g\in C_{\Delta}(E)$
\begin{equation}\label{eqn2}
B_1g(x)=\int (g(\xi)-g(x)) \nu_B(x,d\xi)+m(x)g(x)
\end{equation}
for some non-negative finite kernel $\nu_B$ from $E^\Delta$ to $E^\Delta$ and some $m\in C_\Delta(E)$;
\item\label{it8i} $Q_0 \equiv 0$;
\item\label{it8iii} $Q_1$ is of the form 
\begin{equation}\label{eqn4}
Q_1(g)(x)=\alpha(x)g(x,x), \quad g \in C_\Delta(E) \otimes C_\Delta(E)
\end{equation}
where $\alpha \in C_{\Delta}(E)$ with values in $\mathbb{R}_+$;
\item\label{it8ii} $Q_2$ admits a ($\beta, \pi$)-representation.
Moreover, the parameters $\pi$ and $\beta$  are  bounded and continuous on $(E^\Delta)^2\setminus\{x=y\}$ and $\pi+\overline \pi+2\beta \in \widehat{C}_{\Delta}(E^2)$, where $\overline \pi(x,y)=\pi(y,x)$.
\end{enumerate}

\noindent For the state space $M_+(E)$, the following equivalence holds.\\

\noindent A linear operator $L: \mathcal{D} \to C(M_+(E))$ is $M_+(E)$-polynomial, there exists an $M_+(E)$-valued solution to the martingale problem for all initial conditions in $M_+(E)$ and all solutions have continuous paths, if and only if $L$ satisfies \ref{it8iv}-\ref{it8ii}, $b \in M_+(E)$, and 
$\nu_B$ is a kernel from $E$ to $E$.
\end{theorem}

\begin{remark}\label{rem:equivalence}\phantomsection
\begin{enumerate}
\item 
Condition~\ref{it8v} for the bounded operator $B_1$ is equivalent to 
the requirement that $B_1 g + \| B_1\| g \geq 0$ for all $g \geq0$. This in turn is also equivalent to $(\exp(tB_1))_{t \geq 0}$ being a positive semigroup on $C_{\Delta}(E)$ and to $B_1$ satisfying the positive minimum principle; see Theorem B.II.1.3 in \cite{arendt:86}. 
When $E$ is finite this exactly means that the matrix $B_1$  has non-negative off-diagonal elements. 
The decomposition as of Condition~\ref{it8v} then means that such a matrix is decomposed into a transition rate matrix where the diagonal elements are defined such that the rows sum up to $0$ and a diagonal matrix (corresponding to $m$) where this procedure is compensated.

\item
Note that in the case of $D= C_{\Delta}(E)$, adding 
to $Q_2$ operators of the form
\[
(A \otimes A)g,
\]
where $A$ is a generator of a strongly continuous positive group on $C_\Delta(E)$ does not yield more generality, since $A$ is necessarily of the form $Ag(x)=a(x)g(x)$ for some $a \in C_{\Delta}(E)$ as proved in Lemma \ref{lem:diagonal}.
Therefore, $(A \otimes A)(g)(x,y)=a(x)a(y)g(x,y)$. This however can be absorbed in the function $\beta$.
\end{enumerate}
\end{remark}

\begin{proof}[Proof of Theorem~\ref{mainthm}]
Assume  that $L$ satisfies~\eqref{eq:op} with conditions \ref{it8iv} to \ref{it8ii}. 
By Lemma~\ref{lem:boundedness} condition \ref{it8v}  implies that  $B_1$ satisfying the positive minimum principle on $E^\Delta$. 
Hence, all conditions of Theorem~\ref{main1} are satisfied and we get that $L$ is $M_+(E)$-polynomial and that its martingale problem has an $M_+(E^{\Delta})$ valued solution with continuous paths for every initial condition $\nu\in M_+(E^{\Delta})$. By Theorem~\ref{main1} we also get that any solution with initial value $\nu \in  M_+(E)$ takes values in $M_+(E)$ provided that $b\in M_+(E)$, $\nu_B$ is a kernel from $E$ to $E$, and $m\in C_0(E)$. Indeed, it remains to verify that $\text{bp-}\lim_{n \to \infty}B_1g_n\geq B_1 1$ where
$g_n \in  C_0(E)$ is a sequence such that $\text{bp-}\lim_{n \to \infty}g_n= 1_E$. The claim then follows by the dominated convergence theorem.

We now prove the opposite implication. Assume that $L$ is $M_+(E)$-polynomial, then there exists an $M_+(E^{\Delta})$-valued solution to its martingale problem for each initial condition in $M_+(E^{\Delta})$ and all solutions have continuous paths. Theorem~\ref{Lpol} and Lemma~\ref{IIlem5} imply that $L$ satisfies~\eqref{eq:op}, and due to Lemma~\ref{IIIlem9} also the positive maximum principle on $M_+(E^{\Delta})$. An application of Lemma~\ref{lem3} then yields the result.

Next, assume that $L$ satisfies \ref{it8iv}-\ref{it8ii}, $b \in M_+(E)$,  
and $\nu_B$ is a kernel from $E$ to $E$. By the first part of the theorem it holds that $L$ is $M_+(E)$-polynomial, there exists an $M_+(E^{\Delta})$-valued solution to the martingale problem for all initial conditions in $M_+(E)$, and all solutions have continuous paths. By considering any sequence of functions $g_n\in C_0(E)$ with $0\leq g_n(x)\uparrow 1$ for all $x\in E$ the claim follows by Lemma~\ref{IIIlem8}\ref{it7ii}.

Finally, assume that $L$ is $M_+(E)$-polynomial, there exists an $M_+(E)$-valued solution to the martingale problem for all initial conditions in $M_+(E)$ and all solutions have continuous paths. By the same argument as above $L$ satisfies~\eqref{eq:op} and due to Lemma~\ref{IIIlem9} the positive maximum principle on $M_+(E)$. Conditions  \ref{it8iv}-\ref{it8ii} are then satisfied by Lemma~\ref{lem3}.

It thus remains to prove that $b \in M_+(E)$, $\nu_B$ is a kernel from $E$ to $E$. As a first step, observe that
$$\E[\langle 1,X_t\rangle]\leq \langle 1, X_0\rangle +\int_0^t \langle 1, b\rangle+\|m\|\E[\langle 1 ,X_s\rangle]ds$$
and Lemma~\ref{lem:martingality}\ref{it3i} we can apply the Gronwall inequality to get that by
$$\E[\langle 1,X_t\rangle]\leq (X_0(E^\Delta)+b(E^\Delta)t)\exp(\|m\|t)=:m_1(X_0,t).$$
Next, since $\widetilde B_1g:=B_1g-\|m\|g$ is a bounded dissipative operator on $C_\Delta(E)$ we know that the range of $(\ell-\widetilde B_1)$ is given by $C_\Delta (E)$ for each $\ell>0$. This  implies that the range of $(K-B_1)$ is given by $C_\Delta (E)$ for each $K>\|m\|$. Fix now $K>\|m\|$,  $(h_n)_{n\in \N}\subseteq C_\Delta(E)$ converging to $1_\Delta$ in the bounded pointwise sense, and choose $(g_n)_{n\in \N}\subseteq C_\Delta(E)$ such that $(K-B_1)g_n=h_n$. We then have that
\begin{align*}
\langle g_n,X_0\rangle e^0
&=-\E[\langle g_n,X_t\rangle]e^{-Kt}-\int_0^t \E[ -K\langle g_n,X_s\rangle e^{-Ks}+(\langle g_n,b\rangle+\langle B_1g_n,X_s\rangle) e^{-Ks}] ds\\
&=-\E[\langle g_n,X_t\rangle]e^{-Kt}+\frac 1 K \langle g_n,b\rangle (e^{-Kt}-1) -\int_0^t \E[-\langle h_n,X_s\rangle e^{-Ks}] ds.
\end{align*}
Since $|\E[\langle g_n,X_t\rangle]|\leq \|g_n\|\E[\langle 1,X_t\rangle]\leq\|g_n\|m_1(X_0,t)$ sending $t$ to infinity we get
$$
\langle g_n,X_0\rangle
=-\frac 1 K \langle g_n,b\rangle+\int_0^\infty \E[\langle h_n,X_s\rangle e^{-Ks}] ds,
$$
and hence
$$\lim_{n\to\infty}\langle g_n,X_0\rangle=-\frac 1 K b(\Delta)+\int_0^\infty\E[X_s(\Delta)]e^{-Ks}ds=-\frac 1 K b(\Delta),$$
 showing that $\lim_{n\to\infty} g_n(x)=b(\Delta)=0$ for each $x\in E$. Using that 
$$|\langle g_n,X_0\rangle|\leq \sup_{n\in \N}\|h_n\| \int_0^\infty \E[m_1(X_0,s)e^{-Ks}] ds<\infty$$
we can conclude that $(g_n)_{n\in \N}$ is a bounded sequence. 

Next, observe that $\lim_{n\to\infty} (B_1g_n(x)-Kg_n(x))=\lim_{n\to\infty} -h_n(x)=-1_{\{x=\Delta\}}\leq 0$ for each $x\in E^\Delta$. Inserting the from of $B_1$ yields
$$(\lim_{n\to\infty}g_n(x))\Big(\nu_B(x,\Delta)-1_{\{x=\Delta\}}(\nu_B(\Delta, E^\Delta)+K-m(\Delta))\Big)=-1_{\{x=\Delta\}}.$$
Since inserting $x=\Delta$ we get that $(\lim_{n\to\infty}g_n(x))\neq 0$, we can conclude that $\nu_B(x,\Delta)=0$ for each $x\in E$ proving the claim.
\end{proof}

\begin{example}
In accordance with Example~\ref{ex1}, 
suppose that $L$ satisfies the conditions of Theorem~\ref{mainthm} for $B_0=0$ and $Q_1=0$. Then the dual operator $L_m$ satisfies $L_m^0=0,\ldots, L_m^{m-1}=0$ and
$$L_m^m (g^{\otimes m})
= mB_1(g)\otimes g^{\otimes(m-1)}+\frac {m(m-1)}2 Q_2(g\otimes g)\otimes g^{\otimes (m-2)}.$$
Observe that $L_m^m$ satisfies positive minimum principle. Since $B_1$ and $Q_2$ are bounded the operator $L_m^{m}$ generates a strongly continuous positive semigroup and the moment formula can thus be applied.

A quick look at the generator $L$ of the process $(S_t\mu)_{t\geq0}$ studied in Example~\ref{ex1} shows that $L$ satisfies the conditions of Theorem~\ref{mainthm} for $B_0=0$,$B_1=0$, $Q_1=0$, and
$Q_2(g)=\sigma^2 g$
and is thus of the given form for $C=0$, $n=1$, and $A_1g=\sigma g$.

A similar reasoning can be applied to compute the generator  $L$ of $(Y_t\mu)_{t\geq0}$ for $dY_t=\sigma\sqrt Y_t dW_t $. The obtained operator $L$ does not satisfy the conditions of Example~\ref{ex1} but  satisfies the conditions of Theorem~\ref{mainthm} for the parameters $B_0=0$, $B_1=0$, $Q_1(g)(x)=\sigma^2g(x,x)$, and
$Q_2=0$.
\end{example}

\subsection{Uniqueness in law}

Having established existence to the martingale problem, we are now concerned with uniqueness. As in the finite dimensional case, uniqueness in law follows if the (marginal) moments determine 
the finite-dimensional marginal distributions. This property is known as \emph{determinacy} of the moment problem. Whenever this holds true, we can then conclude uniqueness in law by relying on the moment formula.
The results of this section generalize Theorem 4.2 in  \cite{FL:16} and provide conditions ensuring the existence of exponential moments. These conditions are satisfied by measure-valued affine diffusions treated in Section \ref{sec:affine} below.

\begin{lemma}
Let $L$ be an $M_+(E)$-polynomial operator and  $(X_t)_{t \geq 0} $  an  $M_+(E^{\Delta})$-valued solution to the corresponding martingale problem with initial value $X_0=\nu$.
Suppose that the conditions of Theorem~\ref{thm1} are satisfied  for each $g\in \bigoplus_{k=0}^m D^{\otimes k}$ and each $m\in \N$. If for each $t\geq 0$ there exists $\varepsilon >0$ with
\begin{equation}\label{ineq:mgf}
\mathbb{E}[e^{\varepsilon\langle 1,X_t\rangle}]<\infty,
\end{equation}
then the law of $X$ is uniquely determined by $L$ and $\nu$.
\end{lemma}

\begin{proof}
Observe that condition \eqref{ineq:mgf} yields $\mathbb{E}[e^{\varepsilon|\langle g,X_t\rangle|}]<\infty$ for each $g\in D$ with $|g|\leq 1$.
Following proof of Lemma~4.1 in \cite{FL:16} we thus get that the law of 
$$(\langle g_1,X_t\rangle,\ldots,\langle g_n,X_t\rangle)_{t\geq 0}$$
is uniquely determined by $L$ and $\nu$ for each $g_1,\ldots, g_n\in D$. Using that $D$ is a dense subspace of $C_\Delta(E)$ the claim follows.
\end{proof}

If $L$ admits a representation as described in Theorem~\ref{main1} explicit sufficient conditions for \eqref{ineq:mgf} can be provided. Indeed, the following theorem assumes a linear growth condition, meaning that $Q_2\equiv 0$, to ensure existence of exponential
moments.

\begin{theorem}\label{thm5}
Let $L:\Dcal\to C(M_+(E))$ be a linear operator satisfying the assumptions of Theorem~\ref{main1} and $X$ be an $M_+(E^\Delta)$-valued solution of the corresponding martingale problem. 
Assume that $Q_2\equiv 0$ and that $B_1:D\to C_\Delta(E)$ is the generator of a strongly continuous semigroup $(P_t)_{t\geq 0}$. 
Then for each $t\geq 0$ there exists $\varepsilon >0$ such that condition \eqref{ineq:mgf} holds true and thus the martingale problem for $X$ is well-posed.
\end{theorem}

The proof of Theorem~\ref{thm5} relies on the following  generalization of Theorem~1.3 in \cite{H:85}.

\begin{lemma}\label{lem1}
Let $(Y_t)_{t\geq 0}$ be a stochastic process satisfying $dY_t=b^Y_t dt+dM_t$ for some continuous local martingale $M$ whose quadratic variation is given by $d[M]_t=(\sigma^Y_t)^2 dt$ and some continuous stochastic processes $b^Y$ and $\sigma^Y$. Suppose that $\sup_{t\geq0}|b^Y_t|\leq m$ and $\sup_{t\geq0}|\sigma^Y_t|\leq \rho$ almost surely for some constants $m,\rho$. Then for any non-decreasing convex function $\Phi$ on $\R$ it holds
$$\E[\Phi(Y_t)]\leq \E[\Phi(V)],$$
where $V$ is a Gaussian random variable with mean $Y_0+mt$ and variance $\rho^2 t$.
\end{lemma}
\begin{proof}
Without loss of generality we can assume the existence of an auxiliary Brownian motion $W$ independent of $Y$. Set then
$$Z^\pm_t:=Y_0+mt+M_t\pm\int_0^t(\rho^2-(\sigma^Y_t)^2)^{1/2}dW_t$$
and observe that
$\frac 1 2 (Z^+_t+Z^-_t)-Y_t=mt-\int_0^t b^Y_s ds\geq0.$
Moreover, since
$$[Z^\pm]_t=[M]_t+\int_0^t(\rho^2-(\sigma^Y_s)^2)ds=\rho^2t,$$
by L\'evy characterization theorem we get that $(\frac 1 \rho(Z^{\pm}_t-Y_0-mt))_{t\geq 0}$ is a standard Brownian motion. This in particular implies that $Z_t^+$ and $Z^-_t$ are both Gaussian random variables with mean $Y_0+mt$ and variance $\rho^2 t$.
We can thus use the properties of $\Phi$ to conclude 
$$\E[\Phi(Y_t)]\leq \E[\Phi(\frac 1 2 (Z^+_t+Z^-_t))]\leq \frac 1 2 \E[\Phi(Z^+_t)+\Phi(Z^-_t)]=\E[\Phi(V)].$$
\end{proof}

We are now ready to prove Theorem~\ref{thm5}. The proof is an adaptation  of the proof of Lemma C.1 in \cite{FL:16} to the current setting.  

\begin{proof}[Proof of Theorem~\ref{thm5}]
Fix $T\geq 0$ and note that setting $f_s:=\langle P_{T-s}1, \,\cdot\, \rangle$, by \eqref{martprob} we have that
$$N^{f_s}_t = {f_s}(X_t) - {f_s}(\nu) - \int_0^t L{f_s}(X_r) dr$$
is a local martingale for each $s\in[0,T]$. By Lemma~\ref{IIlem5} we know that 
$$[N^{f_{s_1}},N^{f_{s_2}}]_t=2\int_0^t \Gamma(f_{s_1},f_{s_2})(X_r)dr
=\int_0^t \langle \alpha (f_{s_1})\alpha (f_{s_2}),X_t \rangle dt.$$
Set then $Y_t:=\langle P_{T-t}1,X_t \rangle - A_t$ for
$
A_t :=\langle P_T1,X_0 \rangle + \int_0^t \langle P_{T-s}1,b\rangle ds.
$
Since
$$d\langle P_{T-t}1,X_t \rangle=-\langle B_1P_{T-t}1,X_t \rangle dt
+ L f_t(X_t) dt+(\text{local martingale})$$
we know that $Y$ is a local martingale whose quadratic variation is given by
 $$d[Y]_t = \langle \alpha (P_{T-t}1)^2,X_t \rangle dt=:\sigma_t^2 dt.$$
Using that $\sigma_t^2
\leq(\sup_{t\in[0,T]} \|\alpha\|\| P_{T-t}1\|)\langle P_{T-t}1,X_t \rangle$ and the fact that $A$ is deterministic and thus bounded on $[0,T]$ we get
\begin{equation}\label{ineq:alp}
\sigma_t^2
\leq C(1+|Y_t|).
\end{equation}
Following now the proof of Lemma C.1 in \cite{FL:16} we can apply Lemma~\ref{lem1} to conclude that
$$\mathbb{E}[e^{\varepsilon |Y_T|}]<\infty$$
for some $\e>0$.
The claim follows by noting that   $\langle 1,X_T \rangle=Y_T$.
\end{proof}

\section{Affine diffusions on $M_+(E)$}
\label{sec:affine}

We now define operators of \emph{affine type} and consider solutions to the associated martingale problems. As we will show these solutions constitute a subclass of polynomial processes, which correspond to classical measure-valued branching Markov diffusions analyzed e.g.~in~\cite{F:88, E:00, L:10} and thus range in the class of measure-valued affine processes. They can be considered as 
generalizations of the Dawson-Watanabe superprocess (in the terminology of~\cite{E:00}\footnote{Note that in \cite{L:10}  ``Dawson-Watanabe superprocess'' is used for a class of measure-valued branching processes which can also exhibit jumps. Essentially the class of processes that we obtain here is the subset of processes with continuous trajectories therein.}), also called super-Brownian motion,  where the constant diffusion coefficient is replaced by a function and where the spatial motion is governed by an operator satisfying the positive minimum principle. As it is well-known for affine processes, additionally to the moment formula, the Laplace transform of the process' marginals is exponentially affine in the initial state and the characteristic exponent can be computed by solving a Riccati partial differential equation, which we shall introduce in Section \ref{sec:Ric}.

\subsection{Operators of affine type and their characterization in the diffusion setting}

Let us here start by defining operators of affine type, inspired by the classical form of the infinitesimal generator of affine processes on finite dimensional state spaces, see, e.g.,  \cite{KW:71, DFS:03, CFMT:11, CKMT:16}.

\begin{definition}\label{affineop}
We say that a linear operator  $L:\Dcal\to \Mcal(M(E))$ is of \emph{affine type} on $M_+(E)$ if there exist maps $F: D \to \mathbb{R}$ and $R: D \to C_{\Delta}(E)$
such that 
\[
L \exp(\langle g, \fdot \rangle )(\nu)= (F(g) + \langle R(g), \nu \rangle) \exp(\langle g, \nu \rangle )
\]
for all $g \in D_-$ and $\nu \in M_+(E)$, where $D_-$ denotes all function in $D$ with values in $\mathbb{R}_-$.
\end{definition}

For a finite-dimensional diffusions it is well-known that its generator is of affine type if and only if the drift and diffusion coefficients are affine functions (see, e.g.,  \cite{KW:71, DFS:03}). The following assertion states the same result for the measure-valued setting. We report and prove it here for the reader convenience.
It implies in particular that (in the current diffusion setting) operators of affine type constitute a subclass of polynomial operators.

\begin{theorem}\label{Laffine}
A linear operator $L:\mathcal{D}\to \Mcal(M(E))$ is of affine type on $M_+(E)$ and its carr\'e-du-champs operator $\Gamma$ is an $M_+(E)$-derivation if and only if $L$ admits a representation
\begin{equation}\label{eq:op_affine}
\begin{split}
Lf(\nu) = &B_0(\partial f(\nu)) + \langle B_1(\partial f(\nu)),\nu\rangle\\
& + \frac{1}{2} \Big(Q_0(\partial^2 f(\nu)) + \langle Q_1(\partial^2 f(\nu)), \nu\rangle \Big), \quad f\in \Dcal,  \nu\in M_+(E)
\end{split}
\end{equation}
for some linear operators $B_0:D\to\R$, $B_1:D\to  C_\Delta(E)$, $Q_0:D\otimes D\to\R$, $Q_1:D\otimes D\to C_\Delta(E)$. The restriction of these operators on $D_-$ and $D_-\otimes D_-$, respectively, are uniquely determined by $L$ and the following relations hold for the functions $F$ and $R$ of Definition \ref{affineop}
\begin{align}
F(g)=B_0(g) +\frac{1}{2} Q_0(g \otimes g),\label{eq:F}\\
R(g)=B_1(g)+ \frac{1}{2} Q_1(g \otimes g) \label{eq:R}.
\end{align}
\end{theorem}

\begin{proof}
Assume first that $L$ satisfies \eqref{eq:op_affine}. Then 
for $f(\nu)=\exp(\langle g, \nu\rangle)$ with $g \in D_-$ and $\nu \in M_+(E)$ we have
$$
L \exp (\langle g, \fdot \rangle)(\nu) = \exp (\langle g, \nu \rangle) (B_0(g) + \langle B_1(g), \nu \rangle   + \frac{1}{2}( Q_0(g \otimes g) + \langle Q_1(g \otimes g), \nu \rangle)). 
$$
Hence $L$ is an affine operator with $F(g)=B_0(g)+\frac{1}{2}Q_0(g \otimes g)$ and $R(g)= B_1(g)+ \frac{1}{2}Q_1(g \otimes g)$. Hence $L$ is of affine type. 
Moreover, a direct calculation yields for all $p, q \in \mathcal{D}$ 
\[
\Gamma(p,q)(\nu) = Q_0\left(\partial p(\nu)\otimes \partial q(\nu)\right)+
\left\langle Q_1(\partial p(\nu)\otimes \partial q(\nu)), \nu\right\rangle\qquad\text{for all $\nu\in M_+(E)$,}
\]
which is an $M_+(E)$-derivation.

Conversely, assume that $L$ is of affine type and that its carré-du-champs operator $\Gamma$ is an $M_+(E)$-derivation.
We now explain how the affine property and the fact that $\Gamma $ is a derivation imply the form \eqref{eq:op_affine} on $f(\nu) =\exp(\langle g, \nu \rangle)$ for $g\in D_-$. The extension to all of $\mathcal{D}$ follows then similarly as in the proof of Theorem \ref{Lpol}.
Observe that since $L$ satisfies the affine property it holds that
\[
L \exp(\langle g, \nu \rangle )= (F(g) + \langle R(g), \nu \rangle) \exp(\langle g, \nu \rangle ),
\]
for each $g\in D_-$.
Define then
\begin{align*}
B_0(g)
&:=2F(g)-\frac 1 2 F(2g),&
 B_1g&:= R(g)-\frac 1 2 R(2g),\\
 Q_0(g\otimes g)
&:=F(2g)-2F(g),&
 Q_1(g\otimes g)&:= R(2g)-2R(g),
\end{align*}
and observe that the claim follows by showing that these operators are linear. Set $p_g(\nu):=\exp(\langle g,\nu\rangle)$ and note that
\begin{align*}
B_0(g)+\langle B_1(g),\nu\rangle&=L(2p_gp_g(\nu)^{-1}-\frac 1 2 p_g^2p_{g}(\nu)^{-2})(\nu),\\
Q_0(g\otimes g)+\langle Q_1(g\otimes g),\nu\rangle&=L(p_g^2p_{g}(\nu)^{-2}-2p_gp_g(\nu)^{-1})(\nu),
\end{align*}
for each $g\in D_-$. Fix now $g_i,f_i\in D_-$ and $\alpha_i,\beta_i\in \R$ such that $\sum \alpha_i g_i=\sum \beta_i f_i$ and note that
\begin{align*}
&\sum \alpha_i B_0(g_i)-\sum \beta_i B_0(f_i)
+\langle \sum \alpha_i B_1(g_i)-\sum \beta_i B_1(f_i),\nu\rangle
=Lp(\nu)
\end{align*}
for $p=\alpha_i (2p_{g_i}p_{g_i}(\nu)^{-1}-\frac 1 2 p_{g_i}^2p_{g_i}(\nu)^{-2})
-\beta_i (2p_{f_i}p_{f_i}(\nu)^{-1}-\frac 1 2 p_{f_i}^2p_{f_i}(\nu)^{-2})$. Since $\partial p(\nu)=0$ and $\partial p^2(\nu)=0$ the linearity of $B_0$ and $B_1$ follows by Lemma~\ref{lem2}. Similarly, 
fix $g_i,f_i\in D_-$ and $\alpha_i,\beta_i\in \R$ such that $\sum \alpha_i g_i\otimes g_i=\sum \beta_i f_i\otimes f_i$ and note that
\begin{align*}
&\sum \alpha_i Q_0(g_i\otimes g_i)-\sum \beta_i Q_0(f_i\otimes f_i)
+\langle \sum \alpha_i Q_1(g_i\otimes g_i)-\sum \beta_i Q_1(f_i\otimes f_i),\nu\rangle
=Lp(\nu)
\end{align*}
for $p=\alpha_i ( p_{g_i}^2p_{g_i}(\nu)^{-2}-2p_{g_i}p_{g_i}(\nu)^{-1})
-\beta_i (p_{f_i}^2p_{f_i}(\nu)^{-2}-2p_{f_i}p_{f_i}(\nu)^{-1})$. Since $\partial p(\nu)=0$ and $\partial p^2(\nu)=0$ the linearity of $Q_0$ and $Q_1$ follows by Lemma~\ref{lem2}.
\end{proof}

We  now introduce the subclass of \emph{measure-valued affine diffusions} via affine type operators.

\begin{definition}
Let $L$ be of affine type on $M_+(E)$. Then any continuous $M_+(E)$-valued solution to the martingale problem for $L$ is called a
\emph{measure-valued affine diffusion}.
\end{definition}

\subsection{Laplace transform and Riccati (partial) differential equations}\label{sec:Ric}
Due to the additional properties of affine diffusions, not only the moment formula holds true but also the Laplace transform can be computed explicitly via non-linear partial differential equations of Riccati type. We start here by introducing the corresponding solution concept. Recall that we  refer to differential equations as ODEs even though they often correspond to PDEs or PIDEs depending on the state space and the involved operators.

\begin{definition}\label{def:solaffine}

Let $R$ be given by \eqref{eq:R}. Then we call  a function $t \mapsto \psi_t$ with values in $D$ a solution to the Riccati ODE
\[
\partial_t \psi_t =  R(\psi_t) , \quad \psi_0=g \in D,	
\]
if for every $t >0$ it holds
\begin{align}\label{eq:ODERic}
\langle \psi_t, \nu \rangle =\langle g, \nu \rangle + \int_0^t \langle R(\psi_s), \nu \rangle ds
\end{align}
for all $\nu \in M_+(E^{\Delta})$.
\end{definition}

Note that the reason why \eqref{eq:ODERic} is called Riccati differential equation is because $g \mapsto R(g)$ is a quadratic function as seen from \eqref{eq:R}.

\begin{remark}
Similarly as for the moment formula note that 
the above solution concept reduces to a more classical solution if we take $\nu=\delta_{x} $ with $x\in E$, Indeed,  \eqref{eq:ODERic}  can then be transformed into
\begin{align*}
\psi_t(x)= g(x)+ \int_0^t R(\psi_s)(x) ds
\end{align*}
and thus reduces to a classical (except of the integral form) solution of a Riccati PDE. 

Note that in contrast to strong solutions, which were in the current context for instance considered in \cite[Appendix]{I:86} or \cite[Section 7.1]{L:10} and where $\psi$ is required to be a continuously (Fr\'echet) differentiable
curve $\psi_t: \mathbb{R}_+ \to D$ , 
Definition \ref{def:solaffine} corresponds to an (analytically) weak solution concept. 

\end{remark}

In the following we prove the exponential affine property of the Laplace transform.

\begin{theorem}\label{thm_affine}
Let $L$ be an operator of affine type with functions $F$ and $R$. Suppose that an $M_+(E^{\Delta})$-valued solution $(X_t)_{t \geq 0} $ to the martingale problem for $L$ exists.
Assume furthermore that there is a solution  in the sense of Definition \ref{def:solaffine} of the Riccati ODE on $[0,T]$ given by
\begin{equation}\label{ODEaffine}
\partial_t \psi_t =  R(\psi_t), \qquad 
\psi_0=g \in D_{-},				
\end{equation}
which takes values in $D_{-}$ where $D_{-}$ denotes all functions in $D$ with values in $\mathbb{R}_{-}$. Define furthermore
\begin{equation}\label{eq:phi}
\phi_t=\int_0^t F(\psi_s)ds
\end{equation}
and suppose it takes values in $\mathbb{R}_-$.
Then, for all $0\leq t\leq T$ the representation
\begin{equation}\label{eq:Laplacetrans}
\mathbb{E}\left[\exp( \langle g, X_T \rangle) \,\big|\, \mathcal{F}_t \right ]= \exp(\phi_{T-t}+ \langle \psi_{T-t}, X_t \rangle)
\end{equation}
 holds almost surely. 
\end{theorem}

\begin{proof}
Due to the fact that $X$ solves the martingale problem for $L$, we can deduce 
analogously as in the proof of \cite[Theorem 7.13]{L:10} that $\exp(\phi_{T-t} + \langle \psi_{T-t}, X_t \rangle) $ is a local martingale. By the assumptions that $\psi$ takes values in $D_-$ and $\phi \in \mathbb{R}_-$, this is actually a true martingale as it is bounded by $1$ and we thus have
\[
\exp(\phi_{T-t} + \langle \psi_{T-t}, X_t \rangle)=\mathbb{E}[\exp(\phi_0+\langle \psi_{0}, X_T \rangle) | \mathcal{F}_t]=\mathbb{E}[\exp(\langle g, X_T \rangle) | \mathcal{F}_t],
\]
proving the claim.
\end{proof}

\subsection{Existence and uniqueness in law}

We can now combine Theorem \ref{Laffine}, Theorem \ref{main1} and Theorem \ref{thm_affine} to get the following existence and uniqueness in law result for measure-valued affine diffusions. Recall that $D\subseteq C_\Delta(E)$ be a dense linear subspace containing the constant function~$1$.

\begin{corollary}\label{mainaff}
Let  $L:\Dcal\to C(M_+(E))$ be a linear operator of form~\eqref{eq:op_affine}, where
\begin{enumerate}
\item\label{it4ivaf} $B_0: D \to \mathbb{R}$ is given by $B_0(g) =\langle g,b\rangle$ with $b \in M_+(E^{\Delta})$;
\item\label{it4vaf} $B_1: D \to C_{\Delta}(E)$ satisfies the positive minimum principle on $E^{\Delta}$;
\item\label{it4iaf} $Q_0 \equiv 0$;
\item\label{it4iiaf} $Q_1$ is of the form 
\[
Q_1(g)= \alpha \diag(g), \quad \text{ that is,} \quad  Q_1(g)(x)=\alpha(x)g(x,x), \quad g \in D \otimes D
\]where $\alpha \in C_{\Delta}(E)$ with values in $\mathbb{R}_+$.
\end{enumerate}
Then $L$ is of affine type on $M_+(E)$ and its martingale problem has an $M_+(E^{\Delta})$-valued solution with continuous paths for every initial condition $\nu\in M_+(E^{\Delta})$.

If additionally, the measure $b$ and the initial condition $\nu$ lie in $ M_+(E)$
and if there exists some function $m \in C_{\Delta}(E)$ such that
 $$\text{bp-}\lim_{n \to \infty}(B_1(1-g_n)-m(1-g_n))\leq 0,$$
   for some sequence 
   $g_n \in D \cap C_0(E)$ satisfying $\text{bp-}\lim_{n \to \infty}g_n= 1_E$,
then any solution to the martingale problem takes values  in $M_+(E)$.

If, in either of the $M_+(E)$ or $M_+(E^{\Delta})$ cases,
\begin{align}\label{eq:Ricconcrete}
\partial_t \psi_t= B_1 \psi_t + \frac{1}{2}\alpha \psi_t^2, \quad \psi_0=g \in D_{-}
\end{align}
has a solution 
in the sense of Definition \ref{def:solaffine} with values in $D_{-}$ for all $t \geq 0$, then the corresponding martingale problem is well-posed.
\end{corollary}

\begin{proof}
The existence result is just a consequence of Theorem \ref{Laffine} and Theorem \ref{main1}. Concerning the well-posedness of the martingale problem, 
observe that the function $F$ associated to the operator $L$ of affine type is of the form
\[
F(g)=B_0(g)= \langle g, b\rangle
\]
for $b \in M_+(E^{\Delta})$. Hence, $F$ maps $\psi_t \in D_{-}$ to $\mathbb{R}_-$ and 
 the assumptions of Theorem \ref{thm_affine} are satisfied. Hence the Laplace transform of $X_t$ is  given by 
\[
\mathbb{E}[\exp(\langle g, X_t \rangle)]= \exp(\phi_t+ \langle \psi_t, X_0 \rangle).
\]
Since $ g \in D_{-}$ was arbitrary and $D$ is dense in $C_{\Delta}(E)$, the law of 
$X_t$ is
uniquely determined for all $t \geq 0$. From \cite[Theorem 4.4.2]{EK:09}, we
infer that $X$ is a Markov process  thus the unique solution to the martingale problem associated to $L$.
\end{proof}

\begin{remark}\label{rem:discussion}
Let us here explain how the above result relates to the literature on $(\xi,\phi)$- (Dawson-Watanabe) - superprocesses as defined in \cite[p.42]{L:10}. Indeed, $\xi$ is the spatial motion which is described by its infinitesimal generator $\mathcal{A}$ that can be obtained from the current $B_1$  via 
\[
\mathcal{A}g=B_1g - m g,
\]
as explained in Remark \ref{rem1}. The branching mechanism $\phi$ (not to be confused with the function defined via \eqref{eq:phi}) then corresponds to
\[
\phi(x,g)=-m(x)g(x)+\frac{1}{2}\alpha(x) g(x)^2.
\]
Hence, in the notation of \cite[Eq.(2.27)]{L:10}, $b=-m$ and $c=\frac {1}{2}\alpha$. The $B_0$-part can be added as an immigration part as considered in \cite[Section 9.3]{L:10}.

This identification allows us to apply the results of \cite{L:10} to our setup. In view of the Riccati equations this means in particular that \eqref{eq:Ricconcrete} corresponds (up to a change of sign) to 
\cite[Eq.(7.4)]{L:10}, which is of the form
\begin{align*}
\partial_tV_tf&= \mathcal{A}V_tf- \phi(\fdot, V_tf)
=\mathcal{A} V_tf + m V_tf-\frac{1}{2} \alpha (V_tf)^2=B_1 V_tf -\frac{1}{2} \alpha (V_tf)^2,\\
V_0f &= f (=-g \text{ in our notation}).
\end{align*}
Supposing that $\mathcal{A}$ generates a Feller semigroup $(P_t)_{t\geq0}$ and that $m=-b \in C_{\Delta}(E)$, then by Theorem \cite[Theorem 7.11]{L:10}  this equation is in turn equivalent to the integral equation \cite[Eq.(5.32)]{L:10} 
\[
V_tf= P_tf- \int_0^t  P_{t-s}\phi(\fdot, V_sf) ds,
\]
which corresponds to a mild solution.
As $\psi_t=-V_t f$ with initial value $\psi_0=g=-f$ and denoting the positive semigroup generated by $B_1$ with $Q$ we thus get a similar equivalence between our solution concept for \eqref{eq:Ricconcrete} and the following integral equation in weak form (corresponding to a weakly mild solution)
\begin{align*}
\langle \psi_t, \nu \rangle&= \langle Q_t g, \nu \rangle+ \frac{1}{2}\int_0^t \langle Q_{t-s} \alpha \psi_s^2, \nu \rangle ds,\\
&=\langle P_t g, \nu \rangle+ \int_0^t \langle P_{t-s} (m \psi_s + \frac{1}{2}\alpha \psi_s^2), \nu \rangle ds.
\end{align*}
By \cite[Corollary 5.17]{L:10} which is a consequence of \cite[Proposition 2.20 and (Eq. 2.33)]{L:10}  it admits a unique negative solution in the domain of $\mathcal{A}$ (in the strong sense as generator of the Feller semigroup $(P_t)_{t\geq 0}$). The latter property is a consequence of \cite[Theorem 7.11]{L:10}. From this discussion we can now deduce the following corollary.
\end{remark}

\begin{corollary}\label{cor:solRic}
Let the Conditions \ref{it4ivaf}-\ref{it4iiaf} of Corollary \ref{mainaff} be satisfied and suppose additionally that $B_1$ is the generator of a strongly continuous positive semigroup.
Then \eqref{eq:Ricconcrete} has a unique solution in the sense of Definition \ref{def:solaffine} with values in $D_-$ for all $t \geq 0$ and the martingale problem is well-posed.
\end{corollary}

\begin{proof}
Since $B_1 $ is the generator of a strongly continuous positive semigroup, it can be decomposed into $B_1g=\mathcal{A}g+ mg$ where $m \in C_{\Delta}(E)$ and
 $\mathcal{A}$ is the (strong) generator of a Feller semigroup. Hence the assertion follows
 from Theorem 7.11 and Corollary~5.17 in \cite{L:10} as explained in Remark \ref{rem:discussion}. 
\end{proof}

\begin{remark}
Note that the conditions of Corollary \ref{cor:solRic} are in line with the ones of Theorem \ref{thm5} from which uniqueness in law can be deduced as well.
\end{remark}

Similarly as in the polynomial case we get a characterization of affine diffusions when $D=C_{\Delta}(E)$.

\begin{corollary}\label{mainthmaff}
Let $D=C_\Delta(E)$ and let $L:\Dcal\to C(M_+(E))$ be a linear operator.
Then $L$ is  of affine type on $M_+(E)$, there exists an $M_+(E^{\Delta})$-valued solution to the martingale problem for all initial conditions in $M_+(E^\Delta)$ and all solutions have continuous paths, if and only if $L$ satisfies \eqref{eq:op} with
\begin{enumerate}
\item\label{it8ivaf} $B_0: C_{\Delta}(E) \to \mathbb{R}$ is given by $B_0(g) =\langle  g,b\rangle$ with a Radon measure $b \in M_+(E^{\Delta})$;
\item\label{it8vaf} For every $g\in C_{\Delta}(E)$
\begin{equation}\label{eqn2af}
B_1g(x)=\int g(\xi)-g(x) \nu_B(x,d\xi)+m(x)g(x)
\end{equation}
for some non-negative finite kernel $\nu_B$ from $E^\Delta$ to $E^\Delta$ and some $m\in C_\Delta(E)$;
\item\label{it8iaf} $Q_0 \equiv 0$;
\item\label{it8iiiaf} $Q_1$ is of the form 
\begin{equation}\label{eqn4af}
Q_1(g)(x)=\alpha(x)g(x,x), \quad g \in C_\Delta(E) \otimes C_\Delta(E)
\end{equation}
where $\alpha \in C_{\Delta}(E)$ with values in $\mathbb{R}_+$.
\end{enumerate}

\noindent For the state space $M_+(E)$, the following equivalence holds.\\

\noindent A linear operator $L: \mathcal{D} \to C(M_+(E))$ is of affine type on $M_+(E)$, there exists an $M_+(E)$-valued solution to the martingale problem for all initial conditions in $M_+(E)$ and all solutions have continuous paths, if and only if $L$ satisfies \ref{it8ivaf}-\ref{it8iiiaf}, $b \in M_+(E)$, and 
$\nu_B$ is a kernel from $E$ to $E$.

Moreover, in either case 
\[
\partial_t \psi_t= B_1 \psi_t + \frac{1}{2}\alpha \psi_t^2, \quad \psi_0=g \in C_{\Delta}(E)
\]
has a unique solution  in the sense of Definition \ref{def:solaffine} with values in $D_{-}$ for all $t \geq 0$, which implies \eqref{eq:Laplacetrans} and well-posedness of the martingale problem.
\end{corollary}

\begin{proof}
All assertions except the very last one follow from combining Theorem \ref{Laffine} and Theorem \ref{mainthm}. The last one is a consequence of Corollary \ref{cor:solRic} since $B_1$ is  decomposed into $B_1g=\mathcal{A}g+mg$ where $\mathcal{A}g(x)=\int g(\xi)-g(x) \nu_{B}(x,d\xi)$  generates a Feller semigroup  corresponding to a pure jump process and $m \in C_{\Delta}(E)$.
\end{proof}

\section{Examples and applications}\label{sec:applications}

Having established the general setting for measure-valued affine and polynomial diffusions, we now consider several important examples. 

\subsection{Finite underlying space}

We start with 
a finite underlying space $E$ consisting of $m \geq 1$ points, i.e~ $E=\lbrace 1,...,m \rbrace$. Then $C_\Delta(E)=C(E)$ is finite-dimensional, hence any dense linear subspace is equal to the whole space.  In this setting, any $M_+(E)$-valued process is of the form $X_t=\sum_{i=1}^m Z_t^i\delta_i$ for some $\mathbb{R}^m_+$-valued process $Z=(Z^1,...,Z^m)$.
Hence, Theorem~\ref{mainthm} corresponds to the finite dimensional characterization of polynomial operators with state space $\mathbb{R}^m_+$, as proved in \cite[Proposition 6.4]{FL:16}.
Indeed, comparing with this proposition we recognize the following identifications:
\begin{enumerate}
\item the Radon measure $b\in M_+(E)$ corresponds to the vector $\beta_J \in \mathbb{R}^m_+$;
\item the operator $B_1$ corresponds to the submatrix $B_{JJ}^{\top}$  and \eqref{eqn2} translates to  positive off-diagonal elements;
\item the function $\alpha\in C_{\Delta}(E)$ with values in $\mathbb{R}_+$ corresponds to the vector $\phi\in \mathbb{R}^m_+$;
\item the functions $\pi$ and $\beta$ correspond to the matrix $\Pi$ and $\alpha\in\mathbb{S}^m$ while Condition \eqref{eq:matrix} translates to $\alpha +\operatorname{Diag}(\Pi^Tx)\operatorname{Diag}(x)^{-1}\in \mathbb{S}^m_+$ for all $x\in \mathbb{R}^m_{++}$.

\end{enumerate}

\subsection{Underlying space $E \subseteq \mathbb{R}$}

We consider here the case $E^{\Delta}=[a,b]\subseteq \mathbb{R}^\Delta$, for some $- \infty \leq a< b \leq \infty$ and set 
\[
D:= \{ f|_E \, : \, f \in C_c^{\infty}(\mathbb{R}) + \mathbb{R} \}.
\]
Recall that we identify $E^{\Delta}$ with $E$, when $E$ is compact corresponding to the case $a > -\infty$ and $b < \infty$.
Our goal is to analyze Theorem \ref{main1} in this setting.  We therefore start by considering generators of strongly continuous positive groups on $C_{\Delta}(E)$ such that their domain contains both $D$ and $A(D)$. For results on generators of strongly continuous positive groups we refer to Appendix \ref{app:C}, in particular 
to Definition \ref{def:admissible} and Definition~\ref{def:iso} for the notions of \emph{admissibility} and \emph{lattice isomorphisms}. 
In the following lemma we denote by $C_{\Delta}^k(E)$ the restriction of $C^k(\mathbb{R})$-functions to $E$, i.e.~$C_{\Delta}^k(E)=\{ f|_E \, : \, f \in C^k(\mathbb{R})\cap C_{\Delta}(\mathbb{R})\}$.

\begin{lemma}\label{lem:Ashape} Let $V$ be a lattice isomorphism on $C_\Delta(E)$, which additionally leaves 
$C_{\Delta}^1(E)$ and $C^2_{\Delta}(E)$-functions invariant.
Moreover, let $h \in C^1_{\Delta}(E)$ and 
$\tau \in C_{\Delta}^1(E)$ be an admissible function in the sense of  Definition \ref{def:admissible}. 
Define $\delta_{\tau}$ via \eqref{eq:fshape}.
Then the operator $A$ given by
\begin{equation}\label{eq:Ashape}
A:=V^{-1}\delta_{\tau} V + h, 
\end{equation}
satisfies the conditions of Theorem \ref{IIL:KKT2}. That is, $A$ is the generator of a strongly continuous positive group of $C_\Delta(E)$ and its domain contains both $D$ and $A(D)$.
\end{lemma}
 
\begin{proof} 
First of all note that the domain of $A$ is given by $\{g \in C_{\Delta}(E) \, | \, Vg \in D(\delta_{\tau})\}$, where $D(\delta_{\tau})$ is defined in \eqref{eq:derivationdomain}. As $V$ leaves $C_{\Delta}^1(E)$ and $C^2_{\Delta}(E)$-functions invariant and since $ \tau$ and $h$ are in $C_{\Delta}^1(E)$, the domain of $A$ contains $D$ and $A(D)$. Moreover, by Theorem \ref{th:posgroupgen} the operator $A$ generates a strongly continuous positive group on $C_\Delta(E)$, which yields the result.
\end{proof}

\begin{example}
Let $E =\mathbb{R}$. Set $Vg =pg $ for some strictly positive function $p \in C^{\infty}_{\Delta}(E)$, $ \tau\equiv 1$ and $h \in C^1_{\Delta}(E)$. Then all the conditions of Lemma \ref{lem:Ashape} are satisfied and 
\[
Ag=\frac{1}{p}(p'g+pg')+hg=g'+\widetilde{h} g, \quad g \in D,
\]
where $\widetilde{h}=\frac{p'}{p} +h$. Hence, the stochastic process associated to $A$ is a pure drift process killed at rate $\widetilde{h}$.

To consider some bounded interval, let $E=[0,1]$.
Moreover, let $\tau \in C^1_{\Delta}(E)$ with  such that $\tau(0)=\tau(1)=0$. Then by Remark \ref{rem:admissible} the function
$\tau$ is admissible as it has bounded derivatives. Hence,
\[
Ag=\tau g', \quad g \in D
\]
satisfies the conditions of Lemma \ref{lem:Ashape}.
\end{example}

We can now reformulate Theorem \ref{main1} in the current setting as follows.

\begin{corollary}\label{main1R} 
Let $L:\Dcal\to C(M_+(E))$ be a linear operator of form~\eqref{eq:op}, where $B_0$, $Q_0$ and $Q_1$ satisfy the conditions of  Theorem \ref{main1}. Moreover, for $i=1, \ldots n$, let $A_i$ be of form \eqref{eq:Ashape} such that for $V, \tau$ and $h$ the conditions of Lemma \ref{lem:Ashape} hold. Suppose that
\begin{enumerate}

\item $B_1- \frac{1}{2}\sum_{i=1}^n A_i^2: D \to C_{\Delta}(E)$ satisfies the positive minimum principle on $E^{\Delta}$;

\item $Q_2$ is of the form 
\[
Q_2(g)= C(g)+ \sum_{i=1}^n (A_i \otimes A_i)(g),\quad g \in D \otimes D,
\]
where $C$ admits a ($\beta, \pi$)-representation.
\end{enumerate}
Then conditions \ref{it4iv}-\ref{it4iii} of Theorem \ref{main1} are satisfied.
\end{corollary}

\begin{proof}
This follows directly from Lemma \ref{lem:Ashape}.
\end{proof}

The rest of the section is devoted to the case $E = \mathbb{R}$. 
In view of Corollary \ref{main1R} and Remark \ref{IIrem7}, $B_1$ should satisfy the positive minum principle. Applying Remark \ref{rem1}, we thus decompose $B_1$ into an operator $B_1g=\mathcal{A}g+mg$ where $\mathcal{A}$ satisfies the positive maximum principle and where $m \in C_{\Delta}(\mathbb{R})$. 
It is well-known, that under this condition $\mathcal{A}$ is a Lévy type operator, i.e.
\begin{align}\label{IIeqn24}
\mathcal{A}g = \gamma g' + \frac{1}{2}\Gamma g'' +\int(g( \fdot + \xi) - g - \chi(\xi)g' ) F ( \fdot , d\xi), \quad 
g \in D,
\end{align}
for some continuous functions $\gamma,\, \Gamma$ with $\Gamma \geq 0$, a truncation function $\chi$ , and a kernel
$F ( \fdot, d\xi)$ from $\mathbb{R}$ to $\mathbb{R}$ such that $\int (|\xi|^2 \wedge 1) F ( \fdot , d\xi) < \infty$ . Every operator of this form satisfies $\mathcal{A}1=0$ and the positive maximum principle on $\mathbb{R}$ . The following result expresses Corollary \ref{main1R} within the current setting.

\begin{corollary}\label{main1Rspe}
Let $L:\Dcal\to C(M_+(E))$ be a linear operator of form~\eqref{eq:op}, where $B_0$, $Q_0$ and $Q_1$ satisfy the conditions of  Theorem \ref{main1}. Let $B_1$ be given by
$B_1g=\mathcal{A}g+mg$ where $m \in C_{\Delta}(\mathbb{R})$ and $\mathcal{A}$ is of form $\eqref{IIeqn24}$  with $\Gamma:=\sigma^2+\tau^2$ for some continuous function $\sigma$ and $\tau \in C^1_{\Delta}{(\mathbb{R})}$. Moreover, define $Q_2$ via
\[
Q_2(g)(x,y)=C(g)(x,y)+ \tau(x)\tau(y)g'(x)g'(y)+ h(x)h(y)g(x)g(y),
\]
where $C$ admits a ($\beta, \pi$)-representation and $h \in C^1_{\Delta}(\mathbb{R})$.
Then conditions \ref{it4iv}-\ref{it4iii} of Theorem \ref{main1}  hold true.
\end{corollary}

\begin{proof}
We apply here Corollary \ref{main1R} with $i=1$ and the operator
by $Ag=\tau g' + h$. Note that $\tau$ is admissible as it is bounded and we work on $E=\mathbb{R}$. Hence, we only have to verify that 
$
B_1- \frac{1}{2} A^2
$
satisfies the positive minimum principle on $\mathbb{R}^{\Delta}$. This is the case because of the form of $\Gamma=\sigma^2+ \tau^2$ in $\mathcal{A}$ so that we obtain
\[
B_1g- \frac{1}{2} A^2g=(\gamma+ \tau h) g' +\frac{1}{2} \sigma^2 g'' +\int(g( \fdot + \xi) - g - \chi(\xi)g' ) F ( \fdot , d\xi)+ (m+h^2)g , \quad 
g \in D,
\]
which satisfies the positive minimum principle on $\mathbb{R}$ and also $\mathbb{R}^{\Delta}$ since $D \subset C_c(\mathbb{R})$.
\end{proof}

\begin{remark}
Note that in Corollary \ref{main1Rspe} a non-zero $\tau$ in the specification of $Q_2$ is
coupled with a corresponding diffusive component in the specification \eqref{IIeqn24} of $\mathcal{A}$ and in turn of $B_1$. This is analogous to  \cite[Corollary 6.3 and Example 6.4]{CLS:19} (compare also  \cite[Corollary 6.3 and Example 6.4]{CLS:19}).
\end{remark}

\subsection{Measure-valued Pearson diffusions}

Certain measure-valued polynomial diffusions resulting from Theorem \ref{main1} can be considered as generalizations of one-dimensional Pearson diffusions on $\mathbb{R}_+$. We recall here the definition when their state space is some bounded or unbounded interval of $\mathbb{R}$, see, e.g.,  \cite{FS:08}.

\begin{definition} 
A real-valued stochastic process $(X_t)_{t \geq }$ is called \emph{Pearson diffusion} if it  is a weak solution of the following stochastic differential equation 
$$dX_t=b(X_t)dt+\sigma(X_t)dW_t,$$
where $W_t$ is a standard Brownian motion and $x \mapsto b(x)$ and $x \mapsto \sigma^2(x)$ are polynomials respectively of at most first and second order degree, i.e., one has
\begin{align*}
b(x)&=b_0+b_1x,\\
q(x)&=\frac{1}{2}\sigma^2(x)=q_0+q_1x+q_2x^2,
\end{align*}
for some $b_0,b_1,q_0,q_1,q_2\in \mathbb{R}$. In particular the (extended) generator is defined as 
$$Lg(x)=b(x)g'(x)+q(x)g''(x), \; \; \; \; \; \; g \in C^2(\mathbb{R}).$$
\end{definition}

\begin{remark}
Note that this definition coincides with the notion of one-dimensional polynomial processes with continuous trajectories (see \cite{CKT:12, FL:16}).
\end{remark}

Following \cite{A:20} we can recognize six different types of Pearson diffusions. The statements concerning the invariant measure hold true if the drift is of the form
$b(x)=b_0+b_1x=b_1(x+\frac{b_0}{b_1})$ with $b_1 < 0$ and $b_0 \in \mathbb{R}$.

\begin{enumerate}
\item If $q\equiv q_0 >0$, then $X_t$ is a Ornstein-Uhlenbeck (OU) process with values in $\mathbb{R}$ and the stationary distribution is a Gaussian distribution.
\item\label{it10ii} If $q(x)= q_1 x$  with $q_1 >0$, then $X_t$ is a Cox-Ingersoll-Ross (CIR) process with values in $\mathbb{R}_+$ and the stationary is a Gamma distribution.
\item If $q(x)= q_0+q_1 x+q_2x^2$ with $q_2<0$, then $X_t$ is a Jacobi process with values in a compact interval and the stationary distribution is a (scaled) Beta distribution.
\item\label{it10iv} If $q(x)= q_1 x+q_2x^2$ with $q_1, q_2>0$, then $X_t$ is a Fisher-Snedecor process with values in $\mathbb{R}_+$ and the stationary distribution is a Fisher-Snedecor distribution.
\item\label{it10v} If $q(x)=q_2x^2$ with $q_2>0$, then $X_t$ is a reciprocal Gamma (RG) process with values in $\mathbb{R}_+$ and the stationary distribution is a reciprocal Gamma distribution. Note that this corresponds exactly to diffusion characteristic of the Black \& Scholes model.
\item If $q(x)=q_0+q_1 x+q_2x^2$ with $q_2>0$ and the discriminant $\Delta_q<0$, then $X_t$ is a Student process with values in $\mathbb{R}$ and the stationary distribution is a Student distribution.
\end{enumerate}

Since we are here only interested in (non-negative) measure-valued analogs of the cases \ref{it10ii}, \ref{it10iv} and \ref{it10v} and not in deriving existence of stationary measures, we only distinguish the following  three cases in terms of the characterization given in Theorem~\ref{mainthm}.

\begin{enumerate}
\item[\ref{it10ii}:] As already discussed in Section \ref {sec:affine} the case $Q_2\equiv0$ can be seen as a generalization of Cox Ingersoll Ross processes to the measure-valued case.
\item[\ref{it10iv}:] Taking the general form of Theorem~\ref{mainthm} yields a measure-valued analog of the Fisher-Snedecor process.
\item[\ref{it10v}:] A measure-valued analog of a multivariate Black-Scholes type model corresponds to the case $Q_1= 0$ and $\pi\equiv 0$.
\end{enumerate}

\appendix 

\section{Necessary conditions for the positive maximum principle}\label{app:A}

We here establish the necessary conditions implied by the positive maximum principle for both cases $M_+(E^{\Delta})$ and $M_+(E)$. Recall that $D\subseteq C_\Delta(E)$ is a dense linear subspace containing the constant function~$1$ and let $\Dcal$ be as in \eqref{eq:domain}.

\begin{lemma}\label{lem:necessary}
Let $L:\Dcal\to C(M_+(E))$ be an $M_+(E)$-polynomial operator of form \eqref{eq:op}
satisfying the positive maximum principle on $M_+(E^{\Delta})$. 
Then
\begin{enumerate}
\item\label{it9i}
$B_0$ is a positive linear functional on $D$, hence
there exists a Radon measure  $ b \in M_+(E^{\Delta})$
such that  $B_0(g)= \langle g, b \rangle$;

\item\label{it9ii}  $B_1$ satisfies the positive minimum principle on $E^{\Delta}$;
\item\label{it9iii} 
$Q_0(g\otimes g)=0$ for all $0 \leq g \in D$;

\item\label{it9iv}   
$Q_1(g\otimes g)$ and $Q_2(g\otimes g)$ are non-negative functions for all $g \in D$;

\item\label{it9v}  $ \langle Q_1(g\otimes g), \nu \rangle = 0$ and $\langle Q_2(g\otimes g), \nu^2 \rangle = 0$ for all $\nu \in M_+(E^{\Delta})$ and $0 \leq g \in D$ which are $0$ on the support of $\nu$.
\end{enumerate}
Replacing the positive maximum principle on $M_+(E^\Delta)$ by the positive maximum principle on $M_+(E)$ we obtain that conditions \ref{it9i},\ref{it9iii}, and \ref{it9iv} still hold true and we also get that
\begin{enumerate}[resume]
\item\label{it9iinew} $B_1$ satisfies the positive minimum principle on $E$;
\item\label{it9vnew}  $ \langle Q_1(g\otimes g), \nu \rangle = 0$ and $\langle Q_2(g\otimes g), \nu^2 \rangle = 0$ for all $\nu \in M_+(E)$ and $0 \leq g \in D$ which are $0$ on the support of $\nu$. 
\end{enumerate}

\end{lemma}

\begin{proof}
Consider the following function $f(\nu)=-\langle g, \nu \rangle $ with $g \in D$ such that $g \geq 0$. 
Then $\sup_{M_+(E^{\Delta})} f=f(0) = 0$. By the positive maximum principle for $L$, it follows that
\[
L f(0)= -B_0(g)-\langle B_1 g, 0 \rangle \leq 0. 
\]
Hence, $B_0(g) \geq 0$ and since $g$ was arbitrary we conclude that $B_0$ is a positive linear functional on $D$. 
By assumption $D$ is a linear subspace of $C_\Delta(E)$  that contains an interior point of the cone of non-negative function, namely the constant function $1$. In this case  \cite[Theorem and Corollary 2 of Chapter V.5.4]{sch:99} (alternatively \cite[Section 2.3.1]{B:03}) yields that every positive linear map on $D$ can be extended to a positive linear map on $C_\Delta(E)$. 

Let now $g$ be as above with the additional property that $g(x)=0$ for some $x \in E^{\Delta}$, i.e., $\inf_{E^{\Delta}} g =g(x) = 0$ and $\sup_{M_+(E^{\Delta})} f = f(n \delta_x)= 0$ for every $n \in \mathbb{N}$. By the positive maximum principle for $L$,  we have
\[
L f(n \delta_x)= -B_0(g)-\langle B_1 g, n\delta_x \rangle \leq 0. 
\]
Hence 
\[
-\langle B_1 g, \delta_x \rangle  \leq \frac{B_0(g)}{n}
\]
and sending $n \to \infty$, we conclude that $B_1g (x)\geq 0$, which implies the positive minimum principle on $E^{\Delta}$.\\

Consider now the function $f(\nu)=- (\langle g, \nu\rangle)^2$ for $g \in D$. As it attains its maximum  at $0$ with $\partial f(0)=0$ and $\partial^2 f(0)= -2 g \otimes g$, we have again by the positive maximum principle
\[
L f(0)=-2 Q_0( g \otimes g)- 2\langle  Q_1( g \otimes g),0 \rangle -2\langle  Q_2( g \otimes g),0^2 \rangle \leq 0,
\]
implying that $Q_0( g \otimes g) \geq 0$.

On the other hand, take a function $\phi \in C_c^{\infty}(\mathbb{R})$ restricted to $\mathbb{R}_+$ of the form $\phi(x)=(x-1)^{2n} $ for $x \in [0,\frac{1}{2}]$ and $n \in \mathbb{N}$ such that it attains is maximum at $0$. Consider $f(\nu)=\phi(\langle g, \nu \rangle)$ for $g  \in D$ with $g \geq 0$. Then $\sup_{M_+(E^{\Delta})}f=f(0)=1$, $\partial f(0)=-2n g$ and $\partial^2f(0)=2n(2n-1) (g\otimes g)$. By the positive maximum principle for $L$, it thus follows that
\[
Lf(0)=-2n B_0(g)+ 2n (2n-1)Q_0(g \otimes g) \leq 0.
\]
Hence,
\[
Q_0(g \otimes g)\leq \frac{B_0(g)}{2n-1}
\]
and sending $n \to \infty$ yields $Q_0(g \otimes g) \leq 0$ for $g \geq 0$. Hence for all $g \geq 0$, $Q_0(g \otimes g)=0$. 

Finally, fix $g \in D$ and $\nu \in M_+(E^{\Delta})$ and consider the function $f(\mu)=- (\langle g, \nu \rangle-\langle n g, \mu \rangle)^2$ for $n \in \mathbb{N}$. Then $f\leq 0$, $f(\frac{\nu}{n})=0$, $\partial f(\frac{\nu}{n})=0$ and $\partial^2 f(\frac{\nu}{n})=-2 n^2 g \otimes g$. Hence, 
\[
Lf\left(\frac{\nu}{n}\right)=-2 n^2\langle Q_1(g \otimes g), \frac{\nu}{n} \rangle -2 n^2\langle Q_2(g \otimes g), \frac{\nu^2}{n^2} \rangle \leq 0.
\] 
and thus 
\[
\langle Q_1 (g \otimes g), \nu \rangle \geq -\frac{\langle Q_2(g \otimes g), \nu^2\rangle}{n},
\]
implying that $\langle Q_1(g \otimes g), \nu \rangle \geq 0$ as $n$ tends to $\infty$. Considering the function $f(\mu)=- (\langle g, \nu \rangle-\langle  \frac{g}{n}, \mu \rangle)^2$, yields by similar arguments that $\langle Q_2(g \otimes g), \nu^2 \rangle \geq 0$.

On the other hand, consider again a function $\phi \in C_c^{\infty}(\mathbb{R})$ restricted to $\mathbb{R}_+$ of the form $\phi(x)=(x-1)^{2n} $ for $x \in [0,\frac{1}{2}]$ and $n \in \mathbb{N}$ such that it attains is maximum at $0$. Fix $0 \leq g \in D $ and  $\nu \in M_+(E^{\Delta})$ such that $g=0$ on the support of $\nu$.  
Consider $f(\mu)=\phi(\langle g, \mu \rangle)$. Then $\sup_{M_+(E^{\Delta})}f=f(\nu)=1$, $\partial f(\nu)=-2n g$ and $\partial^2f(\nu)=2n(2n-1) (g\otimes g)$. By the positive maximum principle for $L$, it thus follows that
\[
Lf(\nu)=-2n B_0(g)-2n\langle  B_1(g), \nu \rangle+ 2n (2n-1)\langle Q_1(g \otimes g),\nu\rangle + 2n (2n-1)\langle Q_2(g \otimes g),\nu^2\rangle \leq 0.
\]
Hence,
\[
\langle Q_1(g \otimes g),\nu\rangle +\langle Q_2(g \otimes g),\nu^2\rangle\leq \frac{B_0(g)+ \langle B_1(g),\nu \rangle}{2n-1}
\]
which implies that 
\[
\langle Q_1(g \otimes g),\nu\rangle +\langle Q_2(g \otimes g),\nu^2\rangle=0.
\]
As both summands are non-negative we obtain 
\[
\langle Q_1(g \otimes g),\nu\rangle=0 \quad \text{and} \quad \langle Q_2(g \otimes g),\nu^2\rangle=0
\]
for all $g \geq 0$ which are $0$ on the support of $\nu$.

The second part of the lemma can be proved analogously.
\end{proof}

\begin{lemma} \label{lem3}
Suppose that $D=C_\Delta(E)$ and suppose that $L$ satisfies the assumptions of Lemma \ref{lem:necessary} with the positive maximum principle on $M_+(E)$ instead of $M_+(E^{\Delta})$. Then  conditions \ref{it8iv}-\ref{it8ii} of Theorem~\ref{mainthm} are satisfied.
\end{lemma}

\begin{proof}
We first apply Lemma~\ref{lem:necessary} to obtain conditions \ref{it9i}, \ref{it9iii}, \ref{it9iv}, \ref{it9iinew}, and \ref{it9vnew} stated there.
Since its domain is given by $C_{\Delta}(E)$, the operator $B_0$ is automatically continuous (see \cite[Theorem V.5.5]{sch:99}) and the Riesz-Markov-Kakutani theorem (e.g.~\cite[Theorem~2.14]{R:87}) yields the representation via a Radon measure $b \in M_+(E^{\Delta})$ given in condition~\ref{it8iv} of Theorem~\ref{mainthm}.

Condition~\ref{it8v} of Theorem~\ref{mainthm} follows by Lemma~\ref{lem:boundedness}. Next, in order to establish the form of $Q_0$ observe that
by linearity condition \ref{it9iii} of Lemma~\ref{lem:necessary} extends to the cone
\[
K_+:=\left\{\sum_{i=1}^k \lambda_i (g_i \otimes g_i)\, | \, k \in \mathbb{N},  \lambda_i \in \mathbb{R}_+, 0 \leq g_i   \in D\right\}.
\]
Since $g=g^+-g^-$ where $0 \leq g^{\pm} \in C_{\Delta}(E)$, 
we have
\[
g \otimes g= \underbrace{2 g^+ \otimes g^+}_{\in K_+} +  \underbrace{2 g^- \otimes g^-}_{\in K_+} - \underbrace{(g^+ + g^-) \otimes (g^+ + g^-)}_{\in K_+}.
\]

This implies that $D\otimes D=K_+-K_+$ and $Q_0$ is therefore a continuous linear functional and thus $0$ on the whole of $D \otimes D$. This yields condition \ref{it8i} of Theorem~\ref{mainthm}.

Next, fix $\nu \in M_+(E)$ and consider
$
C^{\nu}:=\{ g \in C_{\Delta}(E) \, | \, g\equiv 0 \text{ on} \supp(\nu)\}
$.
Then $C^{\nu}$ is a linear subspace of $C_{\Delta}(E)$, which can be represented by $C^{\nu}=C^{\nu}_+ - C^{\nu}_+$, 
where
\[
C^{\nu}_+:=\{ 0\leq g \in C_{\Delta}(E) \, | \, g\equiv 0 \text{ on} \supp(\nu) \}.
\]
Similarly for 
\[
K^{\nu}:=\left \{\sum_{i=1}^k \lambda_i (g_i \otimes g_i) \, | \, k \in \mathbb{N}, \lambda_i \in \mathbb{R}, g_i \in C^{\nu}\right\}
\]
we have $K^{\nu}=K^{\nu}_+ - K^{\nu}_+$, where
\[
K_+^{\nu}:= \left\{\sum_{i=1}^k \lambda_i (g_i \otimes g_i) \, | \, k \in \mathbb{N}, \lambda_i \in \mathbb{R}_+, g_i \in C^{\nu}_+\right\}
\]
Therefore the non-negative linear functionals $g  \mapsto \langle Q_1(g), \nu \rangle$ and $g \mapsto \langle Q_2(g), \nu^2 \rangle$ are continuous and thus $0$ on $K^{\nu}$, proving that 
$$ \langle Q_1(g\otimes g), \nu \rangle = 0\qquad \text{ and }\qquad \langle Q_2(g\otimes g), \nu^2 \rangle = 0,$$
 for all $\nu \in M_+(E)$ and $g \in D$ which are $0$ on the support of $\nu$.
Lemma~\ref{lem:Q1} and Lemma~\ref{lem:copositivity} then yield condition~\ref{it8iii} and  condition~\ref{it8ii} of Theorem~\ref{mainthm}, respectively.
\end{proof}

\begin{lemma}\label{lem:boundedness}
Let $B: C_{\Delta}(E) \to C_{\Delta}(E)$ be a linear operator. Then $B$ satisfies the positive minimum principle on $E$ if and only if 
there is map $m\in C_\Delta(E)$ and a non-negative, finite kernel $\nu_B$ from $E$ to $E^\Delta$ such that
\eqref{eqn2} holds.
 In this case, $B$ is bounded and satisfies the positive minimum principle on $E^\Delta$. Moreover, there is some non-negative (finite) measure $\nu_B(\Delta,\fdot)$ such that \eqref{eqn2} holds also for $x=\Delta$.
\end{lemma}

 \begin{proof}
As explained in Remark~\ref{rem1} setting $\overline Bg(x):=B(g-g(x))(x)$ and $m(x):=B1(x)$ we obtain that $Bg=\overline Bg+mg$ for a map $m\in C_\Delta(E)$ and an operator $\overline B$ satisfying $\overline B1=0$ and the positive maximum principle on $E$. The claim then follows by Lemma~C.2 in \cite{CLS:19}.
\end{proof}

\begin{lemma} \label{lem:Q1}
Let $D\subseteq C_\Delta(E)$ be a dense linear subspace containing the constant function~$1$, and let $Q_1\colon D\otimes D\to  C_{\Delta}(E)$ be a linear operator. Then the following conditions are equivalent:  
\begin{enumerate}
\item \label{eq:positivity}
$\langle Q_1(g \otimes g), \nu \rangle \geq 0$  for all $g \in D$ and  $\nu \in M_+(E)$ with equality if $g \equiv 0$ on the
support of $\nu$.
\item \label{eq:positivity1}  $Q_1(g \otimes g)(x)\geq 0$ for all $g \in D$,  $x \in E$ with equality if 
$g(x)=0$.
\end{enumerate}
Both imply that $Q_1$ is of  form
\eqref{eqn4}
for some function  $\alpha \in C_{\Delta}(E)$ with values in $\mathbb{R}_+$.
\end{lemma}

\begin{proof}
Obviously assertion \ref{eq:positivity1} implies \ref{eq:positivity}. Conversely take $\nu=\delta_x$ to conclude  $Q_1(g \otimes g)(x)\geq 0$ and with equality if $g(x)=0$ as the support of $\delta_x$ is just $\{x\}$.

 For $g \in D \otimes D$ we show now that $Q_1(g)(x)$ depends on $g$  through its values at $(x,x)$. To this end, fix $x \in E$ and note that the map $(g, h)\mapsto Q_1(g\otimes h)(x)$ is bilinear as well as positive semidefinite as  $Q_1(g \otimes g)(x) \geq 0$ for all $g \in D$. Hence it satisfies the Cauchy--Schwarz inequality
\[
|Q_1(g\otimes h)(x)| \le \sqrt{ Q_1(g\otimes g)(x)}\, \sqrt{ Q_1(h\otimes h)(x)}.
\]
This together with \ref{eq:positivity1} implies that $Q_1(g\otimes h)(x)$ depends on $g$ and $h$ only through their values at $x$.

Bilinearity then yields $Q_1(g\otimes h)(x)=g(x)h(x)\alpha(x)$ for some $\alpha:E\to \R$. By non-negativity we also have that $0\leq Q_1(g\otimes g)(x)=g(x)^2\alpha(x)$ and thus that $\alpha(x)\geq0$.
Since $Q_1(g) \in C_{\Delta} (E)$ and as $1 \in D$, $x \mapsto \alpha(x)$ lies necessarily in $C_{\Delta}(E)$. Extending to all of $D \otimes D$ yields \eqref{eqn4}.

\end{proof}

Let us now recall the definition of a positive semidefinite kernel. 

\begin{definition}\label{def:posem}
 A symmetric function $K: (E^{\Delta})^2  \to \mathbb{R} $ is called \emph{positive semidefinite kernel} on $E^{\Delta}$  if
\[
\sum _{i,j=1}^{n}c_{i}c_{j}K(x_{i},x_{j})\geq 0
 \]
holds for any $ x_{1},\dots ,x_{n} \in E^{\Delta}$,  $c_{1},\dots ,c_{n}\in \mathbb{R}$ and $ n \in \mathbb{N}$.
\end{definition}

For our purposes we shall need the bigger set of copositive kernels, see, e.g., \cite{DDFV:16, KV:18}. This generalizes the notion of copositive matrices, i.e.~symmetric matrices $Q \in \mathbb{R}^{n \times n}$ such that $x^{\top}Q x \geq 0$ for all $x \in \mathbb{R}_+^n$.

\begin{definition}
 A symmetric function $K: E^2  \to \mathbb{R} $ is called \emph{copositive  kernel} on $ E$  if
\[
\sum _{i,j=1}^{n}c_{i}c_{j}K(x_{i},x_{j})\geq 0
 \]
holds for any $ x_{1},\dots ,x_{n} \in  E$,  $c_{1},\dots ,c_{n}\in \mathbb{R}_+$ and $ n \in \mathbb{N}$.
\end{definition}

\begin{remark}
In the following lemma copositive kernels naturally arise from the condition $\langle Q_2(g\otimes g), \nu^2 \rangle \geq 0$ for all  $\nu \in M_+(E)$. Note that even in the finite dimensional case a simple characterization of copositive matrices is not available,
see, e.g., \cite{HS:10, V:87}
Indeed, only up to dimension $4$, every copositive matrix can be represented as sum of a positive semidefinite one and a matrix with non-negative entries. For dimensions $n \geq 5$, this is no longer true and copositive matrices are a strict superset thereof for which no simple characterization is known. 

In our case copositivity of $Q_2(g\otimes g)$ (and some further conditions) translate  to the requirements of the $(\beta, \pi)$-representation, where in particular \eqref{eq:matrix} is a rather implicit condition. It can be easily verified when we have the decompostion of $Q_2(g\otimes g)$ into a positive semidefinite kernel and a non-negative function (see Remark \ref{rem:sufficient} below).
\end{remark}

\begin{lemma} \label{lem:copositivity}
Let  $D= C_\Delta(E)$ be a dense linear subspace containing the constant function~$1$, and let $Q_2\colon D\otimes D\to  \widehat{C}_{\Delta}(E^2)$ be a linear operator. Then the following conditions are equivalent:  
\begin{enumerate}
\item\label{it5i} \label{eq:positivity2}
$\langle Q_2(g \otimes g), \nu^2 \rangle \geq 0$  for all $g \in D$ and  $\nu \in M_+(E)$ with equality if $g \equiv 0$ on the
support of $\nu$.
\item\label{it5ii}
For all $g \in D$ the map $(x,y)\mapsto Q_2(g \otimes g)(x,y) $ is a copositive kernel on $ E$ and $ Q_2(g \otimes g)(x,y)=0$  if $g(x)=g(y)=0$.
\end{enumerate}
Both imply that $Q_2$ admits a $(\beta,\pi)$-representation and the corresponding parameters $\pi$ and $\beta$  are  bounded and continuous on $(E^\Delta)^2\setminus\{x=y\}$ and $\pi+\overline \pi+2\beta \in \widehat{C}_{\Delta}(E^2)$, where $\overline \pi(x,y)=\pi(y,x)$;

\end{lemma}

\begin{proof}
We start by showing that \ref{it5i} and \ref{it5ii} are equivalent and first prove that \ref{it5ii} implies \ref{it5i}.
The copositive kernel property of $Q_2(g\otimes g)$ in \ref{it5ii} yields $\langle Q_2(g \otimes g), \nu^2 \rangle \geq 0$ for all $\nu \in M_+(E^{\Delta})$ which follows by approximating $\nu$ weakly via $\sum_{i=1}^n c_i \delta_{x_i}$. If $g \equiv 0$ on the support of $\nu$, then clearly $g(x)=g(y)=0$  for all $x,y \in \supp(\nu)$ and by assumption $Q_2(g \otimes g)(x,y)=0$ as well, which yields $\langle Q_2(g \otimes g), \nu^2 \rangle = 0$.

Conversely, the fact that $\langle Q_2(g \otimes g), \nu^2 \rangle \geq 0$ for all  $\nu \in M_+(E)$ implies that $Q_2(g\otimes g)$ is a copositive kernel. Next we show that $ Q_2(g \otimes g)(x,y)=0$  if $g(x)=g(y)=0$.
Indeed, if $\nu=\delta_x$ and $g\equiv 0$ on the support of $\delta_x$, i.e. $g(x)=0$, we conclude  $Q_2(g \otimes g)(x,x)= 0$. Similarly, take
$\nu=(\delta_x+ \delta_y)$ for $x \neq y \in  E$ to obtain
\[
Q_2(g \otimes g)(x,x)+ 2Q_2(g \otimes g)(x,y)+ Q_2(g \otimes g)(y,y)=\langle Q_2(g \otimes g), \nu^2 \rangle.
\]
If $g\equiv 0$ on the support of $\nu=(\delta_x+ \delta_y)$, then $g(x)=g(y)=0$, which in turn implies together with ~\ref{eq:positivity2} and  $Q_2(g \otimes g)(x,x)=Q_2(g \otimes g)(y,y)=0$ that $Q_2(g \otimes g)(x,y)=0$.
This shows the equivalence of \ref{it5i} and \ref{it5ii}.

To prove form \eqref{eq:copositiveconcrete}, let us first note that if $E$ is a singleton, say $\{x\}$,  \eqref{eq:copositiveconcrete} reduces to 
\[
Q_2(g)(x,x)=K(g)(x,x)=\beta(x,x) g(x,x), \quad g \in D \otimes D
\]
with $\beta(x,x)\geq 0$, 
which is obvious by linearity and non-negativity.

For the general case, fix $x \in  E$ 
and note that the map $(g, h)\mapsto Q_2(g\otimes h)(x,x)$ is bilinear as well as positive semidefinite as  $Q_2(g \otimes g)(x,x)\geq 0$ for all $g \in D$. Hence it satisfies the Cauchy--Schwarz inequality
\[
|Q_2(g\otimes h)(x,x)| \le \sqrt{ Q_2(g\otimes g)(x,x)}\, \sqrt{ Q_2(h\otimes h)(x,x)}.
\]
This together with the fact that $Q_2(g\otimes g)(x,x)=0$ if $g(x)=0$ implies that $Q(g\otimes h)(x,x)$ depends on $g$ and $h$ only through their values at $x$.
Proceeding as in the proof of \ref{lem:Q1} we can conclude that $Q_2(g)(x,x)=\beta(x,x)g(x,x)$ for  all $ g \in D \otimes D$.

Let us now consider $Q_2(g\otimes g)(x,y)$ for $x \neq y$.
Recall that copositivity of $(x,y) \mapsto Q_2(g \otimes g)(x,y)$, means that
\[
\sum_{i,j=1}^n c_i c_jQ_2(g \otimes g)(x_i,x_j)=\sum_{i=1}^n c_i^2 \beta(x_i,x_i) g(x_i)g(x_i)+  \sum_{i \neq j} c_i c_jQ_2(g \otimes g)(x_i,x_j) \geq 0
\]
for any  $ x_{1},\dots ,x_{n} \in  E$,  $c_{1},\dots ,c_{n}\in \mathbb{R}_+$ and $ n \in \mathbb{N}$. 
This together with the fact that $Q_2(g \otimes g)(x,y)=0$ if $g(x)=g(y)=0$ implies that $Q_2(g \otimes g)(x,y)$ can depend on $g$ only via $g(x)^2$, $g(y)^2$ and $g(x)g(y)$. Indeed, suppose that it depended on $g(z_1)g(z_2)$ for some $(z_1,z_2)\neq (x,y)$. Choosing $g$ such that $g(x)=g(y)=0$ and $g(z_i) \neq 0$, yields $Q_2(g \otimes g)(x,y)\neq 0$, a contradiction. If $z_1=x$ and $z_2\neq x,y$, we can always choose $g$ such that the above non-negativity is not satisfied. This proves the claim and we obtain by polarization that the bilinear map $(g,h) \mapsto
Q_2 (g \otimes h)(x,y)$ depends on $g$ and $h$ only through their
values at $x$ and $y$.

Bilinearity and symmetry then yield $Q_2(g\otimes h)(x,y)=(g(x),g(y))^\top A(x,y)(h(x),h(y))$ for some $A(x,y) \in \mathbb{S}^2$. Setting $\pi(x,y)=2A_{11}(x,y)$, $\pi(y,x)=2A_{22}(x,y)$ and $\beta(x,y)=2A_{12}(x,y)$ yields
 the following representation
\[
Q_2(g \otimes h)(x,y)=\frac{1}{2}(\pi(x,y)g(x)h(x)+ \pi(y,x)g(y)h(y)+2\beta(x,y)g(x)h(y)).
\]
Extending this to all $ g \in D \otimes D$ gives \eqref{eq:copositiveconcrete}.

Let us now prove the remaining properties of $\pi$ and $\beta$. By \eqref{eq:copositiveconcrete} we can without loss of generalities set $\pi(x,x)=0$ for each $x\in E$. Indeed, if this is not the case it suffices to replace $\beta(x,x)$ with $\pi(x,x)+\beta(x,x)$.
Concerning the non-negativity of $\pi(x,y)$ for $x\neq y$, choose $g$ such that $g(x) \neq 0$ and $g(y)=0$. Then the above form and copositivity of $(x,y)\mapsto Q_2(g \otimes g)(x,y)$ imply
\begin{align*}
0 &\leq c_1^2Q_2(g \otimes g)(x,x)+c_2^2Q_2(g \otimes g)(y,y)+2c_1c_2 Q_2(g \otimes g)(x,y)\\
&=(c_1^2\beta(x,x)+c_1c_2\pi(x,y))g(x)^2 
\end{align*}
for all $c_1, c_2 \geq 0$. Hence, $\pi(x,y)\geq 0$ and analogously $\pi(y,x)\geq 0$.
Finally, copositivity yields for all $n \in \mathbb{N}$,  $x_1, \ldots, x_n \in  E$ and $c_1, \ldots, c_n \in \mathbb{R}_{++}$, 
\begin{align*}
0 &\leq \sum_{i=1}^n c_ic_jQ_2(g \otimes g)(x_i,x_j)\\
&=\sum_{i=1}^n c_ic_j (\frac{1}{2}\pi(x_i,x_j) g(x_i)^2+\frac{1}{2}\pi(x_j,x_i) g(x_j)^2 + \beta(x_i,x_j) g(x_i)g(x_j))\\
&=(c_1 g(x_1), \ldots, c_n g(x_n)) \left(\beta_n+ \begin{pmatrix} \sum_{j=1}^n \frac{c_j}{c_1} \pi(x_1, x_j) &  &\\
& \ddots &\\
& & \sum_{j=1}^n \frac{c_j}{c_n} \pi(x_n, x_j)
 \end{pmatrix} \right)\begin{pmatrix} c_1 g(x_1)\\ \vdots \\c_n g(x_n) \end{pmatrix},
\end{align*}
where $\beta_n \in \mathbb{S}^{n}$ with entries $\beta_{n,ij}=\beta(x_i,x_j)$.
As this holds for all $ (g(x_1), \ldots, g(x_n))$ and thus by the density of $D$ for all vectors in $\mathbb{R}^n$, Condition \eqref{eq:matrix} follows.

The last regularity condition follows from the fact that $(x,y) \mapsto Q_2(g)(x,y) \in \widehat{C}_{\Delta}(E^2)$ for all $ g \in D \otimes D$.

Finally, assume that $\pi$ is not continuous along the sequence $(x_n,y_n)_{n\in \N}$ converging to $(x,y)\in (E^\Delta)^2$ with $x\neq y$. Without loss of generalities $x_n\neq y_m$ for each $n,m$. Choosing $g\in C_\Delta(E)$ such that $g(x_n)=1$ and $g(y_n)=0$ it suffices to use that $Q_2(g\otimes g)\in C_\Delta(E^2)$ to get a contradiction. Similarly, assuming that $\pi$ explodes along the sequence $(x_n,y_n)_{n\in \N}$ converging to $(x,y)\in (E^\Delta)^2$ one can without loss of generality construct $g(x_n)^4=1/\pi(x_n,y_n)$ and obtain a contradiction. The properties of $\beta$ follow form the fact that $\frac 1 2 (\pi+\overline \pi)+\beta =Q_2(1\otimes 1)\in C_\Delta(E^2)$. The continuity of those maps  guarantees that the parameters of the $(\beta,\pi)$-representation satisfy the needed conditions.
\end{proof}

\begin{remark}\label{rem:sufficient}
Note that \eqref{eq:matrix} is clearly implied if $\beta: (E^{\Delta})^2  \to \mathbb{R}$ is a positive semidefinite kernel. If this is the case, we get the following decomposition of the copositive kernel $(x,y) \mapsto Q_2(g\otimes g)(x,y)$
\begin{align}\label{eq:decomposition}
Q_2(g \otimes g)(x,y)=K(g\otimes g)(x,y)+P(g\otimes g)(x,y),
\end{align}
where we set
\[
K(g)(x,y)=\beta(x,y)g(x,y)\quad \text{and} \quad P(g)(x,y)=\frac{1}{2} (\pi(x,y) g(x,x) + \pi(y,x) g(y,y)).
\]
Then $K$ and $P$ are linear operators on $D \otimes D$. Moreover, since $\beta$ is a positive semidefinite kernel we have
\[
\sum_{i,j} c_i c_j K(g\otimes g)(x_i,x_j)=
\sum_{i,j} c_ic_j \beta(x_i,x_j)g(x_i)g(x_j)=\sum_{i,j} \widetilde{c}_i\widetilde{c}_j \beta(x_i,x_j)\geq 0
\]
where $\widetilde{c}_i=c_ig(x_i)$ for $i=1, \ldots, n$, whence $K(g\otimes g)$ is a positive  semidefinite kernel for all $g \in D$. Moreover, $P(g \otimes g)$ is non-negative, which follows from the non-negativity of $\pi$ and since $\pi(x,x)=0$, we also have $P(g \otimes g)(x,x)=0$.

Expression \eqref{eq:decomposition} is thus a decomposition  a into a positive semidefinite kernel $K(g\otimes g)$ and  a non-negative map $P(g\otimes g): (E^{\Delta})^2 \to \mathbb{R}_+$.

\end{remark}

\section{Existence for martingale problems} \label{app_existence}

This section is dedicated to establish the (essential) equivalence
between the existence of a solution to the martingale problem and the positive maximum principle. 
Here, $E$ is a locally compact Polish space, $D$ a dense linear subspace of $C_\Delta (E)$ containing the constant function $1$, and $L$ a linear operator with domain $\mathcal{D}$ (as defined in~\eqref{eq:domain}) satisfying \eqref{eq:op}. 
The first lemma asserts that the positive maximum principle is implied if a solution to the martingale problem exists.

\begin{lemma}\label{IIIlem9}
If there exists an $M_+(E)$- (or $M_+(E^{\Delta})$ respectively) valued solution $X$ to the martingale problem for $L:\Dcal \to C(M_+(E))$ for each initial condition in $M_+(E)$ (or $M_+(E^{\Delta})$ respectively), then $L$ satisfies the positive maximum principle on $M_+(E)$ (or $M_+(E^{\Delta})$ respectively).
\end{lemma}
The proof of Lemma~\ref{IIIlem9} is well-known and we therefore do not state it here; see for instance  the proof of Lemma~2.3 in \cite{FL:16}.

The next lemma is an adaptation of a classical result from \cite{EK:09}. For the application of this result it is crucial that $L$ is an operator on the space of $C_0$-functions on a locally compact, separable, metrizable space. Since this is not  the case for $M_+(E)$ if $E$ is non-compact, we work on $M_+(E^\Delta)$, which is a locally compact Polish space with respect to the topology of weak convergence which follows from \cite[Section 3.1]{D:93} and \cite[Remark~1.2.3, pages 542-543]{L:70}.

\begin{lemma}\label{IIIlem8}
Suppose that $L$ is of form \eqref{eq:op} and satisfies the positive maximum principle on $M_+(E^\Delta)$. 

\begin{enumerate}
\item\label{it7i} Then for every initial condition in $M_+(E^{\Delta})$, there exists a  continuous $M_+(E^{\Delta})$-valued solution to the martingale problem for $L:\Dcal\to C(M_+(E))$.
\item\label{it7ii}
Let the linear operator $B_0: D \to \mathbb{R}$ be given by $B_0(g)=\langle g, b \rangle$ with $b \in M(E^{\Delta})$. If  $X_0 \in M_+(E)$, $b(\Delta)\leq 0$,
and  $\text{bp-}\lim_{n \to \infty}(B_1(1-g_n)-m(1-g_n))\leq 0$ on $E^\Delta$, where $m\in C_\Delta(\R)$ and $g_n \in D \cap C_0(E)$ is a sequence such that $\text{bp-}\lim_{n \to \infty}g_n= 1_E$, then any solution to the martingale problem takes values  in $M_+(E)$. 
\end{enumerate}
 
\end{lemma}

\begin{proof}
We verify the  conditions of  \cite[Theorem~4.5.4]{EK:09}.
As already mentioned, $M_+(E^\Delta)$ is a  locally compact, separable and metrizable space. Moreover, by Lemma~\ref{IIIlem1mod}, we have that
\[
F^D_c:=F^D(M_+(E^\Delta)) \cap C_c(M_+(E^{\Delta}))
\]
 is a dense subset of $C_0(M_+(E^\Delta))$, to which we restrict the domain of $L$ for the moment.
 Moreover, the positive maximum principle yields that $Lf|_{M_+(E^\Delta)}=Lh|_{M_+(E^\Delta)}$ for all $f,h\in F^D_c$ (extended to $M(E^{\Delta})$) such that $f|_{M_+(E^\Delta)}=h|_{M_+(E^\Delta)}$. Note also that the form of \eqref{eq:op} implies that $L(F^D_c)\subseteq C_0(M_+(E^\Delta))$, so that we may
 regard $L|_{F^D_c}$ as an operator  on $C_0(M_+(E^\Delta))$. 
This means that all the assumptions of in \cite[Theorem~4.5.4]{EK:09} are satisfied.
Define now according to the same theorem the linear operator $L^{\mathfrak{\Delta}}$ on $C(M^{\mathfrak{\Delta}}_+(E^\Delta))$ by 
\[
L^{\mathfrak{\Delta}}f|_{M_+(E^\Delta)}=L((f-f(\mathfrak{\Delta}))|_{M_+(E^\Delta)}), \quad L^{\mathfrak{\Delta}}f(\mathfrak{\Delta})=0
\]
for all $f \in C(M^{\mathfrak{\Delta}}_+(E^\Delta))$ such that $(f-f(\mathfrak{\Delta})_{M_+(E^\Delta)}) \in F^D_c$.

Then \cite[Theorem~4.5.4]{EK:09} yields that for every initial condition in $M^{\mathfrak{\Delta}}_+(E^{\Delta})$, 
there exists a solution $X_t$ to the martingale problem for $L^{\mathfrak{\Delta}}$ with càdlàg sample paths  taking values in $M^{\mathfrak{\Delta}}_+(E^{\Delta})$.
Indeed, we obtain that
\eqref{martprob} is a bounded local martingale (and thus a true martingale) for each $f \in F^D_c$ where $L$ is replaced by $L^{\mathfrak{\Delta}}$. Moreover, by Proposition 2 in \cite{BE:85} we also know that $t\mapsto f(X_t)$ is continuous for each $f \in F^D_c$.

\

Define now $\tau_n:=\inf\{t>0\colon\langle 1,X_t\rangle>n\}$ and  set $\tau_{ \mathfrak{\Delta}}:=\lim_{n\to\infty}\tau_n$. We now aim to show that $\P(\tau_{ \mathfrak{\Delta}}>t)=1$, showing that $\P(X_t\neq  \mathfrak{\Delta})=1$.
Fix $k\in \N$  and $\phi(x):=1+x^k$. Consider an increasing sequence $\phi_m\in C_c^\infty(\R)$ such that $\phi_m(x)=x$ for $|x|\leq m$ and $\phi_m(x)=0$ for $|x|>m+1$. Set
   $f(\nu):=\phi(\langle 1,\nu\rangle)$ and $f_m:=\phi_m\circ f$. By continuity of $t\mapsto f_m(X_t)$, we already know that
  $$\tau_n<\tau_{ \mathfrak{\Delta}}\qquad\text{ and }\qquad\langle 1,X_{\tau_n}\rangle=n$$
   for each $n<m$. Observe then that
$\partial f_m(\nu)=\partial f(\nu)$ and $\partial^2 f(\nu)=\partial^2 f_m(\nu)$ for each $\nu\in M_+(E^\Delta)$ such that $\langle 1, \nu\rangle\leq m$. This implies that
$$f(X_{\tau_n\land t})-f(X_0)-\int_0^{\tau_n\land t}  Lf(X_s)ds$$
is a bounded local martingale for each $n,m$ such that $n,\langle 1, X_0\rangle<m$, and thus a true martingale.

Now, observe that since $L$ is given by \eqref{eq:op} it holds 
$$|Lf|\leq Kf$$
 for some $K>0$. By Fatou lemma  this yields
$$\E[\lim_{n\to\infty}f(X_{\tau_n\land t})  
\leq\lim_{n\to\infty}\E[f(X_{\tau_n\land t})
]
\leq \lim_{n\to\infty}\bigg(f(X_{0})+K\int_0^{ t}  \E[f(X_{\tau_n\land s})] ds\bigg).$$
By the Gronwall inequality we can thus conclude that
$$\E[1+(\lim_{n\to\infty}n1_{\{\tau_n\leq t\}}+\lim_{n\to\infty}\langle 1,X_t  \rangle1_{\{\tau_n> t\}})^k]
=\E[\lim_{n\to\infty}f(X_{\tau_n\land t})
]\leq f(X_0)\exp(Kt)$$
and hence that $\P(\tau_{ \mathfrak{\Delta}}>t)=1$ and 
$
\E[f(X_t  )]\leq f( X_0)\exp(Kt).
$

Finally, we need to prove the local martingale property of \eqref{martprob} for $g\in\Dcal\setminus F^D_c$. Since we already know that $(\tau_n)_{n\in\N}$ increases to infinity, the claim follows by noticing that setting  $g_m:=\phi_m\circ g$ for $m$ large enough and proceeding as before we can prove that the process
$$g(X_{\tau_n\land t})-g(X_0)-\int_0^{\tau_n\land t}  Lg(X_s)ds$$
is a true martingale.

For the second part, set $h_n:=1-g_n$.
We know from Lemma \eqref{lem:martingality} that
\[
\langle h_n,X_t\rangle-\langle h_n,X_0\rangle-\int_0^t B_0h_n+\langle B_1h_n,X_s\rangle ds
\]
is a true martingale. We thus get that
\begin{align*}
\E[\langle h_n,X_t\rangle]&=\E[\langle h_n,X_0\rangle]+\int_0^t \E[B_0h_n+\langle B_1h_n,X_s\rangle] ds\\
&\leq \E[\langle h_n,X_0\rangle]+\E[B_0h_n]t+\int_0^t \langle B_1h_n-mh_n,X_s\rangle+\|m\|\E[\langle h_n,X_s\rangle] ds,
\end{align*}
which by the Gronwall inequality yields
$$\E[\langle h_n,X_t\rangle]\leq a(t)+\int_0^ta(s)\|m\|e^{\|m\|(t-s)}ds,$$
where $a(t)=\E[\langle h_n,X_0\rangle]+B_0h_nt+\int_0^t \langle B_1h_n-mh_n,X_s\rangle ds$. We then get that
$$\E[\langle h_n,X_t\rangle]\leq e^{\|m\|t}\Big(\E[\langle h_n,X_0\rangle]+|B_0h_n|t+\int_0^t \langle |B_1h_n-mh_n|,X_s\rangle ds\Big).$$
By the 
 dominated convergence and the assumptions on $X_0$, $B_0$ and $B_1$ we then get
 \begin{equation}\label{eq:nochargeDelta}
\begin{split}
\E[X_t(\Delta)]&=\E[\lim_{n\to \infty} \langle h_n,X_t\rangle]\\
&\leq \lim_{n\to \infty} e^{\|m\|t}\Big(\E[\langle h_n,X_0\rangle]+|B_0h_n|t+\int_0^t \langle |B_1h_n-mh_n|,X_s\rangle ds\Big)\\
&\leq  e^{\|m\|t}(X_0(\Delta)+ |b(\Delta)|t)\\
&\leq 0.
\end{split}
\end{equation}
Non-negativity of $X_t(\Delta)$ implies that $X_t(\Delta)=0$ a.s., whence $X_t \in M_+(E)$.
Finally, note that a c\`adl\`ag process $X$ on $M_+(E^\Delta)$ such that $X_t(\Delta)=0$ a.s. is c\`adl\`ag also with respect to the topology of weak convergence on $M_+(E)$.
\end{proof}

\section{Generators of strongly continuous positive groups}\label{app:C}

Following \cite{arendt:86}, we recall here the main tools behind the characterization of generators of strongly continuous positive groups. We start by introducing the definitions of $C_\Delta(E)$-derivations, flows and cocycles. The notion of a $C_\Delta(E)$-derivation is similar to Definition~\ref{def:derivation}, where we considered  however bilinear maps.

\begin{definition} An operator $\delta$ on $C_\Delta(E)$ is called $C_\Delta(E)$-derivation if its domain $D(\delta)$ is a subalgebra of $C_\Delta(E)$  containing $1$ such that
$$\delta(fg)=(\delta f)g + f(\delta g)\quad \text{for all } f,g\in D(\delta).$$
Note that this implies $\delta 1=0$.
\end{definition}

\begin{definition} A mapping $\Phi:\mathbb{R}\times E^\Delta \rightarrow E^\Delta$ is called a flow on $E^\Delta$ if the maps $\Phi_t:E^\Delta \rightarrow E^\Delta$ given by $\Phi_t(x)=\Phi(t,x)$ are continuous and satisfy
\begin{align*}
&\Phi_0(x)=x, \; \; \; \; \; \; \; \; \; \; \; x\in E^\Delta,\\
&\Phi_s \circ \Phi_t = \Phi_{s+t},  \; \; \; \; s,t \in \mathbb{R}.
\end{align*}
\end{definition}

It follows from the definition that each $\Phi_t$ is a homeomorphism on $E^\Delta$ and $\Phi_{-t}=\Phi_t^{-1}$. A flow is called continuous if it is continuous with respect to the product topology on $\mathbb{R}\times E^\Delta$. 

\begin{definition}
Given a flow $\Phi$ a family $ (k_t)_{t\in \mathbb{R}}$ is called a cocycle of $\Phi$ if
\begin{align*}\label{eq:cocycle}
&k_0=1\\
&k_{t+s}=k_t\cdot(k_s\circ \Phi_t), \; \; \; \; s,t\in \mathbb{R}.
\end{align*}
\end{definition}

A cocycle $(k_t)_{t\in \mathbb{R}}$ associated to a flow $\Phi$ is called continuous if the mapping $(t,x)\mapsto k_t(x)$ from $\mathbb{R}\times E^\Delta$ into $\mathbb{R}$ is continuous with respect to the product topology on $\mathbb{R}\times E^\Delta$. Let $\Phi$ be a flow and $( k_t)_{t\in \mathbb{R}}$ a cocycle of $\Phi$. Then, for every $t\in \mathbb{R}$, 
\begin{equation}\label{eq:cyclecocycle}
T_tf=k_t\cdot f \circ \Phi_t
\end{equation}
defines a bounded operator $T_t$ on $C_\Delta(E)$,  which satisfies the semigroup property $T_{s+t}=T_sT_t$ for all $s,t\in \mathbb{R}$. Moreover, by \cite[Proposition B-II.3.9]{arendt:86},  if $(T_t)_{t\in \mathbb{R}}$ is a strongly continuous group of positive operators on $C_\Delta(E)$, then there exist a continuous flow $\Phi$ on $E^\Delta$ and a continuous cocycle $(k_t)_{t\in \mathbb{R}}$ of $\Phi$ such that \eqref{eq:cyclecocycle} holds.
In \cite[Section B-II.3]{arendt:86} the  form of the cocycle associated with a positive group is explicitely derived and yields the following characterization, see \cite[Theorem B-II.3.14]{arendt:86}. Note that we work here with $C_\Delta(E)$ instead of $C_0(E)$.

\begin{theorem} \label{th:characterization}

An operator $A$ on $C_\Delta(E)$ is the generator of a positive group $(T_t)_{t\in \mathbb{R}}$ if and only if there exist a $C_\Delta(E)$-derivation which is the generator of a group, a function $h\in C_\Delta(E)$ and $p\in C_\Delta(E)$ satisfying $\inf_{x\in E^\Delta}p(x)>0$ such that
\begin{equation}
A=V\delta V^{-1}+h,
\end{equation}
where $V:C_\Delta(E)\rightarrow C_\Delta(E)$ is given by $Vf = p f$. In that case one has the following representation
\begin{equation}
T_t f (x) = \frac{p(x)}{p(\Phi_t(x))}(\exp\int_0^t h(\Phi(s,x))ds)f(\Phi_t(x)),
\end{equation}
for all $f\in C_\Delta(E)$, $t\in\mathbb{R}$ and $x\in E^\Delta$ and where $\Phi$ satisfies \eqref{eq:cyclecocycle}.
\end{theorem}

\subsection{Case $E\subseteq \mathbb{R}$} 

Fix $E\subseteq \R$ such that  $E^{\Delta}=[a,b]\subseteq \mathbb{R}^\Delta$, for some $- \infty \leq a< b \leq \infty$, where we identify $E^{\Delta}$ with $E$, when  $a > -\infty$ and $b < \infty$. In this case the form of the derivation can be made more explicit.
Indeed, let $\tau:(a,b)\rightarrow \mathbb{R}$ be a continuous function and define $\delta_{\tau}$ by 
\begin{equation}\label{eq:fshape}
\delta_{\tau}g := \begin{cases}
\tau(x)g'(x) & \mbox{if} 	\; \;  \tau(x)\neq 0,\\
0 & \mbox{if} 	\; \;  \tau(x)= 0,
\end{cases}
\end{equation}
for all $x\in (a,b)$ and with domain 
\begin{equation}
\begin{split}\label{eq:derivationdomain}
D(\delta_{\tau})=\{ &g\in C_\Delta(E): g\mbox{ is differentiable in }x\in (a,b)\mbox{ whenever }\tau(x)\neq 0\\
&\mbox{ and there exists a (necessarily unique) }f\in C_\Delta(E) \text{ such that }  \delta_{\tau}g=f \}.
\end{split}
\end{equation}
Note that $\delta_{\tau}$ is a derivation and that the above definition corresponds (up to the modification from  $C[a,b]$ to $C_{\Delta}(E)$) to  \cite[Eq.~(3.25)]{arendt:86} with $m=\tau$ and $\tilde{\delta}_m=\delta_{\tau}$.

For the characterization of generators of positive groups 
on $C_\Delta(E)$ for $E \subseteq \mathbb{R}$  we need two further notions, namely so-called \emph{admissibility} (see \cite[Definition~3.16]{arendt:86}) of the function $\tau$ and the notion of a \emph{lattice isomorphism}.

\begin{definition}\label{def:admissible}
 A function $\tau:(a,b)\rightarrow \mathbb{R}$ is admissible if it is continuous and the following holds: whenever $a\leq c<d\leq b$ such that $\tau(x)\neq 0$ for $x\in (c,d)$ and $\tau(c)=0$ or $c=a=-\infty$ and $\tau(d)=0$ or $d=b=\infty$, then $\int_c^z \frac{1}{|\tau(x)|}dx=\int_z^d\frac{1}{|\tau(x)|}dx=\infty$ for $z\in (c,d)$. Moreover, if  if $a > - \infty $, then  $m(a) =0$ and for $b < \infty$, $m(b)=0$.
\end{definition}

\begin{remark} \label{rem:admissible}
Note that every Lipschitz continuous function (referring to globally Lipschitz in the case of unbounded intervals) is admissible. Note also that $\tau$ only needs to be continuous on the open interval $(a,b)$.
\end{remark}

\begin{definition}\label{def:iso}
A lattice isomorphism is a one-to-one map $V:C_\Delta(E)\rightarrow C_\Delta(E)$ such that $|Vg|=V|g|$ for all $g\in C_\Delta(E)$.
\end{definition} 
 
The following theorem is a refinement of Theorem \ref{th:characterization} when $E^{\Delta}=[a,b]$ and follows from \cite[Theorem B-II.3.28]{arendt:86}. We only adapt it to the current $C_{\Delta}(E)$-setting instead of the $C([a,b])$ setting considered in \cite[Theorem B-II.3.28]{arendt:86}. Note that this only differs when $a=-\infty$ and $b=\infty$.

\begin{theorem}\label{th:posgroupgen} An operator $A$ generates a positive group on $C_\Delta(E)$ if and only if there exist
 a lattice isomorphism $V$ on $C_\Delta(E)$,
 an admissible function $\tau: (a,b)\rightarrow \mathbb{R}$,
 and $h \in C_{\Delta}(E)$ such that
$
A=V^{-1}\delta_{\tau} V + h.
$
\end{theorem}

%\bibliographystyle{abbrvnat}
%\bibliography{annex}

\end{document}